\documentclass[a4paper, 10pt, twoside, notitlepage]{amsart}

\usepackage{amsmath,amscd}
\usepackage{amssymb}
\usepackage{amsthm}
\usepackage{comment}
\usepackage{graphicx, xcolor}

\usepackage{mathrsfs}
\usepackage[ocgcolorlinks, linkcolor=blue]{hyperref}

\usepackage{bm}
\usepackage{bbm}
\usepackage{url}

\usepackage[utf8]{inputenc}
\usepackage{mathtools,amssymb}
\usepackage{esint}
\usepackage{tikz}
\usepackage{dsfont}
\usepackage{relsize}
\usepackage{url}
\usepackage{xcolor}
\usepackage{graphicx}
\usepackage{mathrsfs}
\usepackage[shortlabels]{enumitem}
\usepackage{lineno}
\usepackage{amsmath}
\usepackage{enumitem}
\usepackage{amsthm} 
\usepackage{verbatim}
\usepackage{dsfont}
\numberwithin{equation}{section}

\allowdisplaybreaks

\mathtoolsset{showonlyrefs}

\graphicspath{{images/}}

\newtheorem{theorem}{Theorem}[section]
\newtheorem{lemma}[theorem]{Lemma}
\newtheorem{proposition}[theorem]{Proposition}
\newtheorem{corollary}[theorem]{Corollary}

\newtheorem{remark}[theorem]{Remark}

\newcommand{\R}{\mathbb R}
\newcommand{\N}{\mathbb N}

\newcommand{\Hs}{H^s(\mathbb R^n)}
\newcommand{\Htil}{\widetilde H^s}
\newcommand{\Bq}{\mathcal B_q}
\newcommand{\K}{\mathcal K}
\newcommand{\supp}{\operatorname{supp}}
\newcommand{\qe}{\text{q.e.}}
\newcommand{\Capa}{\operatorname{Cap}}
\newcommand{\eps}{\varepsilon}
\newcommand{\Lip}{\operatorname{Lip}}

\title[Inverse fractional Obstacle Problem]{An Inverse Obstacle Problem for the Fractional Schrödinger Equation}

\author[G. Uhlmann]{Gunther Uhlmann}
\address{Department of Mathematics, University of Washington, Seattle, WA 98195-4350, USA}
\email{gunther@math.washington.edu}

\author[P. Zimmermann]{Philipp Zimmermann}
\address{Departament de Matem\`atiques i Inform\`atica, Universitat de Barcelona, Barcelona, Spain and Institute of Mathematics, École Polytechnique Fédérale de Lausanne (EPFL), Lausanne, Switzerland}
\email{philipp.zimmermann@epfl.ch}

\begin{document}

		\begin{abstract}
We study an inverse obstacle problem for the fractional Schrödinger operator
$(-\Delta)^s+q$, $0<s<1$.  For each exterior datum, the state is constrained
by a prescribed obstacle in a bounded domain and satisfies the fractional
Schrödinger equation only in the associated noncontact set.  This set is
unknown and depends on the coefficient, so the exterior
Dirichlet-to-Neumann map is nonlinear.  We show that the nonlocal character
of the equation gives a direct way around this moving-free-boundary
difficulty.  Equality of one obstacle measurement on an exterior open set
forces equality of the two corresponding obstacle states in the whole space.
In their common noncontact set one then obtains
$(q_1-q_2)u=0$.  This identity yields recovery of the potential in the exposed
region by measurable unique continuation when $s\in[1/4,1)$, and by the usual
unique continuation principle together with continuity when the potentials
are continuous.  We also prove a geometric coverage theorem for nonnegative
potentials: rational positive scalings of one nontrivial nonnegative exterior
datum expose the whole domain up to a null set.  Consequently, under the
corresponding assumptions, a countable family of nonlinear exterior obstacle
measurements determines the potential globally.
		\medskip
		
		

		\end{abstract}
		
		\maketitle

	\tableofcontents

\section{Introduction}

Inverse problems for nonlocal equations exhibit a phenomenon which has no
direct local analogue: information measured on an exterior open set may
determine the solution throughout the domain.  For the fractional Schrödinger
operator
\begin{equation}
\label{eq: fractional Schrödinger operator}
  L_q \vcentcolon =(-\Delta)^s+q,
  \qquad 0<s<1,
\end{equation}
this exterior unique continuation mechanism is the basis of the fractional
Calderón problem \cite{GSU-fractional}.  Here $(-\Delta)^s$ denotes the
fractional Laplacian, which may be written as the singular integral
\begin{equation}
\label{eq: singular integral frac Lap}
    (-\Delta)^s u(x)
     \vcentcolon =
    c_{n,s}\operatorname{p.v.}
    \int_{\mathbb R^n}
    \frac{u(x)-u(y)}{|x-y|^{n+2s}}\,dy,
\end{equation}
where $c_{n,s}>0$ is chosen so that the Fourier multiplier is
$|\xi|^{2s}$.  The strong unique continuation and Runge approximation
properties in this setting lead to inverse results which are, in several
respects, more flexible than their local counterparts
\cite{calderon2006inverse}. The interested reader may consult the following articles \cite{GSU20,GU2021calder,RZ-unbounded,ghosh2021non,feizmohammadi2021fractional_closed,feizmohammadi2024calder,FKU24,RZ2022LowReg,zimmermann2023inverse,KLW2022,zimmermann2024optimalrungeapproximationdamped,cekic2020calderon,fu2026calderon,GRSU18,RS17,ruland2018exponential,ruland2021single,zimmermann2024-viscous-wave} and the references therein. The purpose of this paper is to show that the
same nonlocal mechanism is also effective for inverse obstacle problems.

Let $\Omega\subset\R^n$ be bounded and open, and let
$\Omega_e=\R^n\setminus\overline\Omega$.  Given an exterior datum
$f\in H^s(\R^n)$ and an obstacle $\psi$, the fractional obstacle state $u$ is
characterized by
\[
  u-f\in \widetilde H^s(\Omega),\qquad
  u\ge \psi \quad\text{in }\Omega
\]
in the quasieverywhere sense, together with the variational inequality
\[
  \mathcal B_q(u,v-u)\ge0
  \qquad\text{for all admissible }v,
\]
where $\Bq$ is the natural bilinear form associated with the fractional Schrödinger operator. Here quasieverywhere means outside a set of $H^s$-capacity zero; the precise
capacitary convention is recalled in Section \ref{sec: obstacle problem}.  In the noncontact set
$D=\{u>\psi\}$ one has $L_qu=0$ in the weak sense, while on
$\Omega\setminus D$ the solution is constrained by the obstacle.  The
exterior measurement considered here is the nonlinear partial
Dirichlet-to-Neumann map
\[
  \Lambda_{q,\psi}(f)|_W \vcentcolon = (-\Delta)^s u|_W,
\]
where $W\subset\Omega_e$ is a nonempty open set.  The inverse problem is to
determine the potential $q$ from such measurements.

This problem is not merely a formal variant of the fractional Calderón
problem.  The obstacle constraint introduces an unknown active set, and hence
the equation is not imposed on a fixed region.  Even for a known obstacle, the
noncontact set can change with the exterior datum and with the potential
itself.  Thus one is trying to recover a coefficient in a region whose
location is also part of the unknown response of the system.

There are several reasons why this inverse problem is natural.  First,
fractional obstacle problems are genuine free-boundary models.  The classical
case $s=1$ describes a solution which is harmonic away from the contact set,
with a free boundary separating the coincidence and noncoincidence regions.
The fractional problem retains this complementarity structure, but the
operator interacts with the whole space.  As emphasized in Silvestre's thesis
and its published version \cite{Silvestre-thesis,Silvestre-obstacle}, the
expected optimal regularity changes from $C^{1,1}$ in the local problem to
$C^{1,s}$ for the fractional Laplacian.  The order of the operator is
therefore visible in the regularity theory at the free boundary.

A second motivation comes from the thin obstacle, or Signorini, problem.  If a
harmonic function in a half-space is prescribed by its trace on the boundary,
then the map from the trace to the normal derivative is $(-\Delta)^{1/2}$.
Consequently a Signorini constraint on the boundary can be reformulated as an
obstacle problem for the half-Laplacian on the boundary itself
\cite[Section 1.2]{Silvestre-thesis}.  Fractional obstacle problems therefore
encode certain local free-boundary problems in one higher dimension, but in a
compressed nonlocal form.  From the inverse point of view, exterior data for
the fractional equation play the role of accessible nonlocal boundary
information, while the free boundary remains hidden inside the domain.

A third motivation is probabilistic.  Up to the usual sign convention,
$(-\Delta)^s$ is the generator of a symmetric stable jump process, and
obstacle problems for $(-\Delta)^s$ arise from optimal stopping.  In
Silvestre's formulation, the value function
\[
  u(x)=\sup_\tau \mathbb E_x[\psi(X_\tau)]
\]
solves a fractional obstacle problem.  Adding a discount factor leads to an
equation of the form $\lambda u+(-\Delta)^s u=0$ away from the stopping
region \cite[Section 1.4]{Silvestre-thesis}.  In financial language, this is
the free-boundary problem behind perpetual American options driven by jump
processes.  The operator $(-\Delta)^s+q$ is a spatially inhomogeneous version
of the same structure: the potential $q$ can be interpreted as a killing or
discounting rate, or more generally as an unknown medium parameter.  The
inverse question asks whether this hidden parameter can be recovered from
exterior observations even though the stopping region is not known.

The local analogue suggests a geometric linearized strategy.  One works near
background solutions with regular free boundary.  Boundary perturbations then
solve a linear Dirichlet problem in the noncontact set with zero data on the
free boundary, while obstacle perturbations prescribe data on the free
boundary or contact set.  A Green identity can identify interior boundary
information, and a Calderón or Runge argument can then recover the coefficient
in the region exposed by the background solution.  This route is natural, but
it requires substantial additional input: free-boundary regularity,
differentiability of the variational inequality, and a nondegeneracy
condition strong enough to turn the critical cone into an effective
zero-trace space on the contact set.

The fractional Schrödinger equation allows a different first step.  Suppose
that two potentials $q_1,q_2$ give the same obstacle measurement for the same
exterior datum $f$ on an exterior open set $W$.  If $u_1,u_2$ are the
corresponding obstacle states, then $u_1-u_2=0$ in $\Omega_e$, and equality of
the measured data gives
\[
  (-\Delta)^s(u_1-u_2)=0\quad\text{in }W.
\]
The exterior antilocality theorem of Ghosh, Salo, and Uhlmann
\cite{GSU-fractional} therefore implies $u_1=u_2$ in all of
$\mathbb R^n$.  Thus one nonlinear measurement determines the whole obstacle
state.  In the common noncontact set $D=\{u_1>\psi\}=\{u_2>\psi\}$ the two
equations may be subtracted, and one obtains
\[
  (q_1-q_2)u=0.
\]
This identity is the basic recovery mechanism of the paper.

To convert this identity into recovery of the coefficient, one must know that
the common state does not vanish too much in the exposed region.  We use two
unique continuation inputs.  For rough $L^\infty$ potentials, the
measurable unique continuation theorem appearing in the single
measurement fractional Calderón reconstruction of Ghosh, Rüland, Salo, and
Uhlmann \cite{GRSU-single} applies in the range $s\in[1/4,1)$.  If such a
measurable UCP theorem is available for another potential class, the same
argument applies.  For continuous potentials, one can instead use open-set
unique continuation together with continuity: if two continuous potentials
differ at a point of the open noncontact set, then they differ on a smaller
open set, so the identity above forces the state to vanish there, and
exterior unique continuation gives a contradiction unless the state is
trivial.  This proves the single-state recovery result, Theorem
\ref{thm:one-state}: one obstacle measurement recovers the potential in the
region exposed by the measured state.

It remains to expose enough of the domain.  We prove that, for nonnegative
potentials, a countable scaled family of exterior data is enough.  Let
$0\le f_0\in C_c^\infty(W_1)$ be nontrivial, and let $h$ be the unconstrained
solution with exterior value $f_0$.  The weak maximum principle and the
strong maximum principle for nonlocal operators \cite{Jarohs-Weth-SMP} imply
that $h>0$ a.e. in $\Omega$.  If $u_t$ denotes the obstacle state with
exterior datum $t f_0$, comparison gives
\[
  u_t\ge t h \quad\text{in }\Omega.
\]
Now suppose that a set of positive measure were never exposed by the states
$u_t$, $t\in\mathbb Q_+$.  On that set one would have $u_t\le\psi$ for every
rational $t>0$, while the comparison inequality gives $t h\le\psi$.  Since
$\psi$ is finite a.e., letting $t\to\infty$ through rational values would
force $h=0$ on a positive-measure set, contradicting strong positivity.  This
is Proposition \ref{prop:coverage}.

Combining the one-state recovery theorem with this coverage result gives the
main uniqueness theorem, Theorem \ref{thm:global}.  In a representative form,
let $W_1,W_2\subset\Omega_e$ be nonempty open sets, let
$0\le f_0\in C_c^\infty(W_1)$ be nontrivial, and let
$\psi\in C_c^\infty(\Omega)$.  If $q_1,q_2\in L^\infty(\Omega)$ are
nonnegative and either $s\in[1/4,1)$ or both potentials are continuous, then
\[
  \Lambda_{q_1,\psi}(t f_0)|_{W_2}
  =
  \Lambda_{q_2,\psi}(t f_0)|_{W_2}
  \qquad \forall t\in\mathbb Q_+
\]
implies $q_1=q_2$ a.e. in $\Omega$.  Thus the nonlinear obstacle
Dirichlet-to-Neumann map is not needed on an open set of boundary amplitudes;
a countable one-parameter family suffices.  The countability is useful
because the coverage argument only discards null sets and then intersects
over the measured family.

Some care is required in formulating the obstacle problem at this level of
generality.  The obstacle constraint is imposed quasieverywhere with respect
to the $H^s$ capacity, the natural noncontact set is quasiopen before
additional regularity is imposed, and the equation in the noncontact set must
be justified by approximation from capacitary sublevel sets.  Section \ref{sec: obstacle problem}
therefore records the analytic foundations explicitly: closedness of the
admissible set, well-posedness, complementarity, localization of the equation
to the noncontact set, the nonlinear exterior DN map, and a
Lewy--Stampacchia-based route to continuity and openness of $\{u>\psi\}$ under
smooth obstacle and exterior-data hypotheses.  This is the point at which the
smoothness assumptions in the main theorem enter: they ensure that the
noncontact set is genuinely open, so that the recovery identity can be read
distributionally on open subsets.

Although the main proof does not require linearizing the obstacle map, we end
with a short remark explaining how the classical free-boundary strategy
changes in the nonlocal setting.  Under additional Mignot-type
differentiability and strict-complementarity hypotheses, exterior
perturbations solve the fractional Schrödinger equation in the exposed
quasiopen set with homogeneous capacitary data on the contact set, while
obstacle perturbations prescribe data on the entire contact set, not merely
on the geometric free boundary.  This conditional linearized picture is
conceptually useful, but it is not used in the proof of the main uniqueness
theorem.

The paper is organized as follows.  Section \ref{sec: obstacle problem} develops the direct obstacle
problem and the nonlinear DN map.  Section~\ref{sec: UCP} states the unique continuation
theorems used later.  Section~\ref{sec: Single state} proves recovery from one obstacle state.
Section~\ref{sec: geometric coverage} proves the comparison, strong-positivity, coverage, and global
uniqueness results.  We conclude with a brief remark on the conditional
linearized analogue of the local inverse obstacle strategy and on how
nonlocality changes the role of obstacle perturbations.

\section{The obstacle problem}
\label{sec: obstacle problem}

In the present section we discuss the well-posedness of the obstacle problem for the fractional Schrödinger operator $L_q=(-\Delta)^s+q$.

\subsection{Functional-analytic setup}
\label{subsec: functional-analytic setup}

Let $\Omega\subset\R^n$, $n\ge 2$, be bounded and open, and set
$\Omega_e=\R^n\setminus\overline\Omega$.  We use the standard \emph{fractional Sobolev space} $H^s(\mathbb R^n)$, $0<s<1$, with the Gagliardo-- Slobodeckij norm
\[
  \|u\|_{H^s(\R^n)}^2
  \vcentcolon =
  \|u\|_{L^2(\R^n)}^2
  +
  [u]^2_{H^s(\R^n)},
\]
where the seminorm $[\,\cdot\,]_{H^s(\R^n)}$ is defined by
\[
    [u]^2_{H^s(\R^n)}\vcentcolon =\iint_{\R^n\times\R^n}
  \frac{|u(x)-u(y)|^2}{|x-y|^{n+2s}}\,dx\,dy.
\]
Moreover, we let $\dot{H}^s(\R^n)$ stand for the \emph{homogeneous fractional Sobolev space}, which is the completion of $C_c^{\infty}(\R^n)$ with respect to $[\cdot]_{H^s(\R^n)}$. Note that Sobolev's embedding yields $\dot{H}^s(\R^n)\hookrightarrow L^{2^*_s}(\R^n)$, where $2^*_s=2n/(n-2s)$; see, for example, \cite{DNPV-Hitchhiker}.

Throughout the paper, unless explicitly stated otherwise, potentials satisfy
\[
  q\in L^\infty(\Omega),\qquad q\ge0\quad\text{a.e. in }\Omega.
\]
For such $q$ we set
\[
  \Bq(u,v) \vcentcolon =\mathcal E_s(u,v)+\int_\Omega q uv\,dx .
\]
Here, $\mathcal E_s$ is the natural energy related to the fractional Laplacian $(-\Delta)^s$, namely
\[
  \mathcal E_s(u,v)
 \vcentcolon =
  \frac{c_{n,s}}2
  \iint_{\R^n\times\R^n}
  \frac{(u(x)-u(y))(v(x)-v(y))}{|x-y|^{n+2s}}\,dx\,dy,
\]
where $c_{n,s}>0$ is the constant from \eqref{eq: singular integral frac Lap}. Indeed, for all $u,v\in C_c^{\infty}(\Omega)$, one has
\[
  \mathcal E_s(u,v)
  =
  c_{n,s}\int_{\Omega}v(x)\,
  \operatorname{p.v.}\!\int_{\R^n}
  \frac{u(x)-u(y)}{|x-y|^{n+2s}}\,dy\,dx
  =
  \int_{\Omega}v(-\Delta)^s u\,dx.
\]
By the fractional Poincar\'e inequality on open bounded sets, one easily sees that $\mathcal E_s(u,u)$ induces an equivalent norm on
\begin{equation}
\label{eq: tilde space 1}
  \Htil(\Omega)\vcentcolon =\overline{C_c^\infty(\Omega)}^{H^s(\R^n)};
\end{equation}
see \cite{fernandez2024integro} or \cite{RZ-unbounded}. We recall that this inequality asserts the existence of a constant $C=C(n,s,\Omega)>0$ such that
\begin{equation}
\label{eq: fractional Poincaré}
    \|u\|_{L^2(\Omega)}\leq C[u]_{H^s(\R^n)},\quad u\in C_c^{\infty}(\Omega).
\end{equation}
In particular, $\Bq$ is continuous on $H^s(\mathbb R^n)\times H^s(\mathbb R^n)$,
and the positivity of $q$, together with \eqref{eq: fractional Poincaré}, gives coercivity on $\Htil(\Omega)$:
\begin{equation}
\label{eq: coercivity estimate}
  \Bq(\varphi,\varphi)\ge c\|\varphi\|_{H^s(\R^n)}^2
  \qquad \varphi\in\Htil(\Omega)
\end{equation}
for some $c>0$.

\subsection{Capacity and quasicontinuous representatives}

We recall the capacitary convention used in the obstacle constraint.  We use
the \emph{Choquet capacity} associated with the Sobolev space $H^s(\mathbb R^n)$.
First, for a compact set $K\subset\mathbb R^n$, set
\[
  \Capa_s(K)
  \vcentcolon =
  \inf\bigl\{\|v\|_{H^s(\mathbb R^n)}^2:\ v\in\mathcal S(\mathbb R^n),\
  v\ge1\ \text{on }K\bigr\},
\]
where $\mathcal S(\R^n)$ denotes the space of Schwartz functions. 

For an open set $G\subset\mathbb R^n$, set
\[
  \Capa_s(G)
  \vcentcolon =
  \sup\{\Capa_s(K):\ K\subset G,\ K\text{ compact}\},
\]
and for an arbitrary set $E\subset\mathbb R^n$, set
\[
  \Capa_s(E)
  \vcentcolon =
  \inf\{\Capa_s(G):\ E\subset G,\ G\text{ open}\}.
\]
The infimum is allowed to be $+\infty$. These definitions are completely analogous to $(s,p)$ Bessel capacities $C_{s,p}$ in
\cite[Definition 2.2.6]{Adams-Hedberg}. Indeed, our space $H^s(\R^n)$, equipped with the Gagliardo--Slobodeckij norm $\|\cdot\|_{H^s(\R^n)}$, is isomorphic to the \emph{Bessel potential space}
\begin{equation}
\label{eq: Bessel potential space}
    H^{s,2}(\R^n) \vcentcolon =\{f\in \mathcal S'(\R^n): f=G_s\ast g\text{ with } g\in L^2(\R^n)\},
\end{equation}
which carries the norm 
\begin{equation}
\label{eq: Bessel norm}
    \|f\|_{H^{s,2}(\R^n)}\vcentcolon = \|g\|_{L^2(\R^n)}\quad\text{for }f=G_s\ast g.
\end{equation}
Above, the symbol $G_s$ stands for the Bessel kernel, which is given by
\[
    G_s\vcentcolon =\mathcal{F}^{-1}((1+|\xi|^2)^{-s/2})\in L^1(\R^n).
\]
More explicitly, it may be computed as
\[
    G_s(x)=\gamma_{n,s}|x|^{(n-s)/2}K_{(n-s)/2}(|x|),
\]
where $1/\gamma_{n,s}\vcentcolon = 2^{s/2-1}(2\pi)^{n/2}\Gamma(s/2)$ and $K_\nu$ is called modified Bessel function of the third kind. Indeed, this representation motivated N.~Aronszajn and K.~T.~Smith to call $H^{s,2}$ Bessel potential space. Clearly, the Bessel potential space $H^{s,2}(\R^n)$ extends to any $s\in\R$ and also other integrability exponents $1\leq p<\infty$. 

Therefore the capacity $\Capa_s$ is equivalent to the $(s,2)$ Bessel capacity $C_{s,2}$. More concretely, there exist $\alpha_1,\alpha_2 >0$ such that
\begin{equation}
\label{eq: equivalence to bessel cap}
    \alpha_1\, C_{s,2}(E)\leq \Capa_s(E)\leq \alpha_2 \,C_{s,2}(E)
\end{equation}
for all $E\subset \R^n$. In particular, $\Capa_s$ and $C_{s,2}$ have the same capacity-zero sets. 

$H^s$ capacities, or more generally Bessel capacities $C_{s,p}$, are outer measures \cite[Sections~1.1 \& 4.7]{evans-mt}, but they are not finitely additive on arbitrary disjoint sets. Indeed they have the following properties.
\begin{theorem}[{\cite[Proposition~2.3.6, Theorem~2.3.11, Proposition~2.3.13]{Adams-Hedberg}}]
\label{thm: properties of cap}
Let $0<s<1$ and $n\geq 1$. Then:
    \begin{enumerate}[(i)]
    \item\label{item: empty set} $\Capa_s(\emptyset)=0$;
    \item\label{item: monotony} $E_1\subset E_2$ $\Rightarrow$ $\Capa_s(E_1)\leq \Capa_s(E_2)$;
    \item\label{item: subadditivity} $\Capa_s$ is $\sigma$-subadditive, that is, for any sequence $(E_k)_{k\in\N}$, one has
    \[
        \Capa_s\left(\bigcup_{k\in\N}E_k\right)\leq \sum_{k\in\N}\Capa_s(E_k);
    \]
    \item\label{item: decreasing seq} If $(K_\ell)_{\ell\in\N}$ is a decreasing sequence of compact sets, then 
    \[
        \Capa_s\left(\bigcap_{\ell\in\N}K_\ell \right)=\lim_{\ell\to \infty}\Capa_s(K_\ell);
    \]
    \item\label{item: increasing seq} If $(E_k)_{k\in\N}$ is an increasing sequence of arbitrary sets, then 
    \[
        \Capa_s\left(\bigcup_{k\in\N}E_k\right)=\lim_{k\to\infty}\Capa_s(E_k);
    \]
    \item\label{item: Borel sets} For any Borel set $B$,
    \[
        \Capa_s(B)=\sup_{K\text{ compact}:K\subset B}\Capa_s(K)=\inf_{G\text{ open}:B\subset G}\Capa_s(G);
    \]
\end{enumerate}
\end{theorem}

Before proceeding, we shall show the following:
\begin{lemma}
\label{lemma: comparison to Lebesgue measure}
    For any $A\subset \R^n$,
    \begin{equation}
    \label{eq: comparision cap and Ln}
    \mathcal{L}^n(A)\leq C\,\Capa_s(A)^{n/(n-2s)}
\end{equation}
where $\mathcal{L}^n$ is the $n$-dimensional Lebesgue measure.
\end{lemma}
\begin{proof}
    First, assume that $K\subset\R^n$ is a compact set and choose $v\in\mathcal S (\R^n)$ such that $v\geq 1$ on $K$. Since $n\geq 2$ and $0<s<1$, the Sobolev embedding $H^s(\R^n)\hookrightarrow L^{\frac{2n}{n-2s}}(\R^n)$ yields
\[
    \mathcal{L}^n(K)^\frac{n-2s}{n}\leq \|v\|^2_{L^{\frac{2n}{n-2s}}(\R^n)}\leq C\|v\|^2_{H^s(\R^n)}.
\]
Thus,
\[
    \mathcal{L}^n(K)^\frac{n-2s}{n}\leq C\Capa_s(K).
\]
By the interior regularity property of the Lebesgue measure, for any Borel set $B$, we have
\[
    \mathcal{L}^n(B)=\sup_{K\text{ compact}: K\subset B}\mathcal{L}^n(K)
\]
and therefore \ref{item: monotony} (or \ref{item: Borel sets}) implies
\begin{equation}
\label{eq: comparision cap and Ln borel}
    \mathcal{L}^n(B)\leq C(\Capa_s(B))^{\frac{n}{n-2s}}.
\end{equation}
Finally, let $E\subset\R^n$ be arbitrary. If $\Capa_s(E)=\infty$, then \eqref{eq: comparision cap and Ln} holds trivially. Otherwise, for fixed $\eps>0$, choose an open set $G\subset \R^n$ with $E\subset G$ and
\[
    \Capa_s(G)<\Capa_s(E)+\eps.
\]
Then \eqref{eq: comparision cap and Ln borel}, \ref{item: monotony} and the properties of the Lebesgue measure give
\[
    \mathcal{L}^n(E)\leq \mathcal{L}^n(G)\leq C(\Capa_s(G))^{\frac{n}{n-2s}}\leq C(\Capa_s(E)+\eps)^{\frac{n}{n-2s}}.
\]
Thus, letting $\eps\to 0$ yields \eqref{eq: comparision cap and Ln} for any $A\subset\R^n$.
\end{proof}

We shall say that a property $P$ holds 
\emph{quasieverywhere}, if it holds
outside a set of $H^s$-capacity zero.  We abbreviate this by saying that $P$ holds ``\emph{q.e.}''. In view of \eqref{eq: comparision cap and Ln}, this notion is a refinement of the statement that a property holds almost everywhere. 

Moreover, we say that a function $f$, defined quasieverywhere in $\R^n$, is \emph{quasicontinuous}, if for each $\eps>0$ there exists an open set $G$ such that
\[
    \Capa_s(G)<\varepsilon
\]
and 
\[
    f|_{\R^n\setminus G}\text{ is continuous.}
\]

\begin{theorem}[{Quasicontinuous representatives, \cite[Theorem~6.2.1]{Adams-Hedberg}}]
\label{thm: quasicontinuous representation}
    Let $n\geq 2$ and $0<s<1$. Then for all $f\in H^s(\R^n)$ the limit
\begin{equation}
\label{eq: quasicontinuous representative}
    \widetilde f(x)\vcentcolon = \lim_{r\to 0}\fint_{B_r(x)}f(y)\,dy
\end{equation}
exists for all $x\in \R^n\setminus E$, where $\Capa_s(E)=0$. Moreover, $\widetilde f$ has the following properties:
\begin{enumerate}[(i)]
    \item\label{item: q convergence} For all $1\leq q\leq 2n/(n-2s)$,
    \[
    \lim_{r\to 0}\fint_{B_r(x)}|f(y)-\widetilde f(x)|^q\,dy=0.
    \]
    \item\label{item: uniform convergence} For any $\eps>0$, there is an open set $G$ such that $\Capa_s(G)<\eps$ and $\widetilde f|_{\R^n\setminus G}$ is continuous.
    \item\label{item: quasicontinuous} $\widetilde f$ is quasicontinuous with $\widetilde f = f$ q.e.~in $\R^n$.
\end{enumerate}
\end{theorem}

\begin{remark}[Uniqueness of quasicontinuous representatives]
\label{rem: Uniqueness of quasicontinuous representations}
    Recall that \cite[Theorem~6.1.4]{Adams-Hedberg} ensures that two quasicontinuous functions $f_j$, $j=1,2$, which coincide a.e. in $\R^n$, satisfy $f_1(x)=f_2(x)$ for q.e. $x\in\R^n$. In particular, the quasicontinuous representative provided by Theorem~\ref{thm: quasicontinuous representation} is unique up to sets of vanishing $H^s$ capacity.
\end{remark}

\begin{remark}[Fine properties of functions]
\label{rem: fine properties}
    If $f$ is an $H^s$ function and $g$ is a.e. equal to $f$, then $g$ is again an $H^s$ function. Hence, if we want to study pointwise properties of $H^s$ functions, we need to consider quasicontinuous representatives. Consequently, all pointwise inequalities below are imposed on these representatives. 
\end{remark}

\begin{lemma}[Stability of quasicontinuous representatives]
\label{lemma: Stability of quasicontinuous representatives}
    If $u_k\to u$ strongly in $H^s(\mathbb R^n)$, then
a subsequence converges to $u$ quasieverywhere.
\end{lemma}
\begin{proof}
    Since $H^s(\mathbb R^n)$ and $H^{s,2}(\R^n)$ are isomorphic, \eqref{eq: equivalence to bessel cap} shows that it is enough to show the assertion for the $C_{s,2}$ Bessel capacity. 

    The latter in turn is a consequence of Propositions~2.3.13 and~2.3.8 in \cite{Adams-Hedberg} as well as \eqref{eq: Bessel potential space}--\eqref{eq: Bessel norm}.
\end{proof}

Furthermore, as in \cite[Definition~6.4.8 and Proposition~6.4.9]{Adams-Hedberg}, we say that a set $E\subset\R^n$ is quasiopen if for any $\eps>0$ there is an open set $G\subset \R^n$ with $\Capa_s(G)<\eps$ such that $E\setminus G$ is open in the relative topology of $G^c$. 

\begin{corollary}
\label{cor: quasiopen}
    Let $n\geq 2$ and $0<s<1$. Let $u\in H^s(\R^n)$ and suppose that $\psi$ is quasicontinuous. Then $\{u>\psi\}$ is quasiopen and Lebesgue measurable after changing $u$ on a set of capacity zero.
\end{corollary}

\begin{proof}
    Up to redefining $u$ on a set of capacity zero, we may assume according to Theorem~\ref{thm: quasicontinuous representation} that $u$ is quasicontinuous. Thus, $v\vcentcolon = u-\psi$ is quasicontinuous. Let $\eps>0$. By definition of quasicontinuity, there is an open set $G\subset\R^n$ such that $\Capa_s(G)<\eps$ and $v|_{\R^n\setminus G}$ is continuous. Therefore
    \[
        \{v>0\}\setminus G
        =
        \{x\in\R^n\setminus G:\ v(x)>0\}
    \]
    is open in the relative topology of $\R^n\setminus G$. Hence $\{v>0\}$ is quasiopen.
    
    Moreover, for each $k\in\N$ one can choose such a set $G_k$ with
    $\Capa_s(G_k)<1/k$.  Then $\{v>0\}\setminus G_k$ is relatively open in
    $\R^n\setminus G_k$, and hence Borel, while
    $\{v>0\}$ differs from this Borel set by a subset of $G_k$. Lemma~\ref{lemma: comparison to Lebesgue measure} gives
    $\mathcal L^n(G_k)\to0$, and therefore $\{u>\psi\}=\{v>0\}$ is Lebesgue measurable after changing $u$ on a set of capacity zero.
\end{proof}
 
If $u\colon\R^n\to \R$ is a measurable function, then we denote the positive and negative parts of $u$ as
\[
    u_+\vcentcolon =\max(u,0)\quad\text{and}\quad u_{-}\vcentcolon =\max(-u,0).
\]
The next lemma studies the regularity properties of $u_+$, when $u$ is an $H^s$ function. Moreover, if $L\colon \R^n\to \R^m$ is a Lipschitz map, then we denote by $\Lip(L)$ the best Lipschitz constant of $L$.

\begin{lemma}[Lattice and Markov properties]
\label{lem:lattice}
Let $\Omega\subset\mathbb R^n$ be an open bounded set and $0<s<1$. Let $T\colon \R\to \R$ be a Lipschitz map.
\begin{enumerate}[(a)]
    \item\label{item: lipschitz map homogeneous} For any measurable $u$ with finite Gagliardo seminorm,
    \[
        [T(u)]_{H^s(\R^n)}\leq \Lip(T)\,[u]_{H^s(\R^n)}.
    \]
    In particular, $T(u)$ has finite homogeneous $H^s$ seminorm whenever $u$ does.
    \item\label{item: lipschitz map inhomogeneous} If $T(0)=0$, then for any $u\in H^s(\R^n)$, we have $T(u)\in H^s(\R^n)$ with
    \[
        \|T(u)\|_{H^s(\R^n)}\leq \Lip(T)\,\|u\|_{H^s(\R^n)}.
    \]
    If, in addition, $u\in\Htil(\Omega)$, then $T(u)\in\Htil(\Omega)$.
\end{enumerate}
Especially, if $a\geq 0$, $b\leq 0$, and $u\in H^s(\R^n)$, then $(u-a)_+\in H^s(\R^n)$ and $(u-b)_{-}\in H^s(\R^n)$. Moreover, if the quasicontinuous representative of $(u-a)_+$ (or $(u-b)_{-}$) vanishes q.e.~on $\R^n\setminus \Omega$, then it belongs to $\widetilde H^s(\Omega)$.

Finally, for every $a\in H^s(\R^n)$ one has
\begin{equation}
\label{eq: lattice energy inequalities}
  \mathcal E_s(a,a_+)\ge \mathcal E_s(a_+,a_+),
  \qquad
  \mathcal E_s(a,a_-)\le -\mathcal E_s(a_-,a_-).
\end{equation}
\end{lemma}

\begin{proof}
 The estimate in \ref{item: lipschitz map homogeneous} is a direct consequence of
\[
  |T(u(x))-T(u(y))|
  \le \Lip(T)|u(x)-u(y)|.
\]
If $T(0)=0$, then also $\|T(u)\|_{L^2(\R^n)}\le \Lip(T)\|u\|_{L^2(\R^n)}$, which gives \ref{item: lipschitz map inhomogeneous}. If $u\in\Htil(\Omega)$, choose
$u_j\in C_c^\infty(\Omega)$ with $u_j\to u$ in $H^s(\mathbb R^n)$.  Then
$T(u_j)\to T(u)$ in $H^s(\mathbb R^n)$, and each $T(u_j)$ has compact support
in $\Omega$. Since $T$ is Lipschitz with $T(0)=0$ the chain rule in Sobolev spaces shows that $T(u_j)\in H^1(\R^n)\subset H^s(\R^n)$. Indeed, $\nabla T(u_j)=T'(u_j)\nabla u_j$ a.e.~in $\R^n$. Thus, by mollification, the assertion follows.

Let $T^{\pm}_a\colon \R\to\R$, $a\in\mathbb R$, be defined by
\[
    T^{\pm}_a(t)\vcentcolon = (t-a)_{\pm},\quad t\in\R.
\]
Then \ref{item: lipschitz map homogeneous} gives the homogeneous seminorm bound for $T_a^\pm(u)$. Moreover, if $a\ge0$, then $(u-a)_+\le u_+\le |u|\in L^2(\R^n)$, and if $a\le0$,
then $(u-a)_{-}\le u^-\le |u|\in L^2(\R^n)$.

The final assertion is a consequence of the capacitary characterization of $\Htil(\Omega)$ as
the subspace of $H^s(\mathbb R^n)$. Indeed, \cite[Theorem~10.1.1]{Adams-Hedberg} (with $E=\R^n\setminus\Omega$, $p=2$, $\alpha=s$) asserts that if $v\in H^s(\R^n)$ then $v\in \widetilde H^s(\Omega)$ if and only if $v=0$ q.e.~in $\R^n\setminus \Omega$. 

It remains to show \eqref{eq: lattice energy inequalities}. For real numbers
$\alpha,\beta$,
\[
  (\alpha-\beta)(\alpha_+-\beta_+)\ge |\alpha_+-\beta_+|^2,
  \qquad
  (\alpha-\beta)(\alpha_- -\beta_-)\le -|\alpha_- -\beta_-|^2.
\]
Applying these pointwise inequalities with $\alpha=a(x)$ and $\beta=a(y)$
and integrating against the fractional kernel gives \eqref{eq: lattice energy inequalities}.
\end{proof}

Next, we discuss a capacitary exhaustion property. To this end, we introduce the following terminology. If $E\subset\R^n$ is any subset of $\R^n$, not necessarily open, then we set
\begin{equation}
\label{eq: tilde space}
    \widetilde H^s(E)\vcentcolon = \overline{\{u\in H^s(\R^n):\supp u\subset E\}}^{H^s(\R^n)};
\end{equation}
see, for example, \cite[Definition~11.5.1]{Adams-Hedberg}. Note that by Theorem~\ref{thm: quasicontinuous representation} the condition $\supp u\subset E$ is well-defined. As shown in \cite[Section~11.5]{Adams-Hedberg}, the spaces \eqref{eq: tilde space} with $E=\Omega$ coincide with $\widetilde H^s(\Omega)$ as defined in \eqref{eq: tilde space 1}.

\begin{lemma}[Capacitary exhaustion property]
\label{lemma: exhaustion property}
   Let $\Omega\subset\mathbb R^n$ be an open bounded set and $0<s<1$. Let $u\in H^s(\R^n)$, $\psi$ be quasicontinuous, and set
   \[
    D \vcentcolon =\{u>\psi\}\quad\text{and}\quad D_m \vcentcolon =\{u-\psi>1/m\}. 
   \]
   Then
\[
  \Htil(D)
  =
  \overline{\bigcup_{m=1}^\infty \Htil(D_m)}^{H^s(\mathbb R^n)}.
\]
Moreover the approximating functions may be chosen bounded.
\end{lemma}

\begin{proof}
    We write $g\vcentcolon =u-\psi$, where $u$ is identified with its quasicontinuous representative. By Corollary~\ref{cor: quasiopen}, the sets
    $D=\{g>0\}$ and $D_m=\{g>1/m\}$ are quasiopen. Moreover, $D_m\subset D_{m+1}\subset D$ and
    \[
        D=\bigcup_{m=1}^{\infty}D_m.
    \]
    The inclusion
    \[
      \overline{\bigcup_{m=1}^{\infty}\Htil(D_m)}^{H^s(\mathbb R^n)}
      \subset \Htil(D)
    \]
    follows immediately from $D_m\subset D$, $m\in\N$, and the monotonicity $\widetilde H^s (A)\subset \widetilde H^s(B)$ whenever $A\subset B$.

    We prove the converse inclusion.  By Adams--Hedberg's compact exhaustion
    theorem for Bessel potential spaces \cite[(11.5.1)]{Adams-Hedberg}, and $H^s(\mathbb R^n)\simeq H^{s,2}(\mathbb R^n)$,
    \[
      \Htil(D)
      =
      \overline{\bigcup_{K\subset D,\ K\text{ compact}}\Htil(K)}
      ^{H^s(\mathbb R^n)} .
    \]
    Hence it is enough to prove that each element of $\Htil(K)$, with
    $K\subset D$ compact, can be approximated in $H^s(\mathbb R^n)$ by bounded
    elements belonging to $\Htil(D_m)$ for suitable $m$.

    Let $v\in\Htil(K)$.  First assume that $v\in L^\infty(\mathbb R^n)$. By
    quasicontinuity of $g$, for every $j\in\mathbb N$ there is an open set $G_j$
    such that
    \[
        \Capa_s(G_j)\le j^{-2}
    \]
    and $g$ is continuous on $\mathbb R^n\setminus G_j$.  By the capacitary
    potential characterization of Bessel capacity, see
    \cite[Proposition~2.3.9 and Theorem~2.3.10]{Adams-Hedberg}, and by the
    equivalence of $\Capa_s$ and $C_{s,2}$, we may choose
    $\eta_j\in H^s(\mathbb R^n)$ such that
    \[
      0\le \eta_j\le 1,\qquad
      \eta_j=1\quad\text{q.e. on }G_j,\qquad
      \|\eta_j\|_{H^s(\mathbb R^n)}\to0 .
    \]
    Passing to a subsequence, we may also assume that $\eta_j\to0$ a.e. in
    $\mathbb R^n$.

    Since $K\setminus G_j$ is compact and $g$ is continuous on $\R^n\setminus G_j$ and strictly
    positive on $K\setminus G_j$, there exists $m_j\in\mathbb N$ such that
    \[
        K\setminus G_j\subset D_{m_j}.
    \]
    Set
    \[
        v_j=(1-\eta_j)v.
    \]
    Since $v,\eta_j\in H^s(\mathbb R^n)\cap L^\infty(\mathbb R^n)$, we have
    $v_j\in H^s(\mathbb R^n)$. Indeed,
    \[
      [\eta_jv]_{H^s(\mathbb R^n)}
      \le
      \|v\|_{L^\infty(\mathbb R^n)}[\eta_j]_{H^s(\mathbb R^n)}
      +\|\eta_j\|_{L^\infty(\mathbb R^n)}[v]_{H^s(\mathbb R^n)},
    \]
    and the $L^2$ part is immediate from boundedness. Moreover,
    $v=0$ q.e. on $\mathbb R^n\setminus K$ by \cite[Theorem~10.1.1]{Adams-Hedberg}, while
    $1-\eta_j=0$ q.e. on $G_j$. Hence, $K\setminus G_j\subset D_{m_j}$ yields $v_j=0$ q.e. on
    $\mathbb R^n\setminus D_{m_j}$, and again by
    \cite[Theorem~10.1.1]{Adams-Hedberg} we have $v_j\in\Htil(D_{m_j})$.

    It remains to show that $v_j\to v$ in $H^s(\mathbb R^n)$. Since
    $v_j-v=-\eta_jv$,
    \[
      \|v_j-v\|_{L^2(\mathbb R^n)}
      \le \|v\|_{L^\infty(\mathbb R^n)}
      \|\eta_j\|_{L^2(\mathbb R^n)}
      \to0.
    \]
    Moreover,
    \[
    \begin{split}
      [\eta_jv]_{H^s(\mathbb R^n)}^2
      &\le
      2\|v\|_{L^\infty(\mathbb R^n)}^2[\eta_j]_{H^s(\mathbb R^n)}^2  \\
      &\quad
      +2\iint_{\mathbb R^n\times\mathbb R^n}
      \eta_j(y)^2
      \frac{|v(x)-v(y)|^2}{|x-y|^{n+2s}}\,dx\,dy .
    \end{split}
    \]
    Since $\|\eta_j\|_{H^s(\R^n)}\to 0$, the first term tends to zero. The second term tends to zero by dominated
    convergence, because $0\le\eta_j\le1$, $\eta_j\to0$ a.e., and
    $v\in H^s(\mathbb R^n)$. Thus $v_j\to v$ strongly in
    $H^s(\mathbb R^n)$.

    Finally, let $v\in\Htil(K)$ be arbitrary. The truncations
    $T_\ell(t)=\max\{-\ell,\min\{t,\ell\}\}$, $\ell>0$, are Lipschitz and satisfy $T_\ell(0)=0$. By Lemma
    \ref{lem:lattice} and \cite[Theorem~10.1.1]{Adams-Hedberg}, $T_\ell(v)\in H^s(\mathbb R^n)$ with
    $T_\ell(v)=0$ q.e. on $\mathbb R^n\setminus K$. Hence
    $T_\ell(v)\in\Htil(K)$ by \cite[Theorem~10.1.1]{Adams-Hedberg}. Moreover,
    $T_\ell(v)\to v$ strongly in $H^s(\mathbb R^n)$: the $L^2$ convergence is
    dominated convergence, and the Gagliardo seminorm convergence follows from
    dominated convergence applied to
    $v-T_\ell(v)$, since
    \[
        |[v-T_\ell(v)](x)-[v-T_\ell(v)](y)|\le 2|v(x)-v(y)|.
    \]
    Applying the preceding
    argument to $T_\ell(v)$ and using a diagonal sequence gives bounded
    approximants from $\bigcup_m\Htil(D_m)$. This proves the reverse inclusion
    and the assertion about bounded approximating functions.
\end{proof}

\subsection{Well-posedness of the obstacle problem}

\begin{lemma}[Closedness of the admissible set]
\label{lem:closed-convex}
Let $\Omega\subset\mathbb R^n$ be an open bounded set and $0<s<1$. Let
$f\in H^s(\mathbb R^n)$ and suppose $\psi$ is quasicontinuous in $\Omega$.
Assume that
\begin{equation}
\label{eq: admissible set}
  \K_{f,\psi}
  \vcentcolon =
  \{v\in H^s(\mathbb R^n):\ v-f\in\Htil(\Omega),\
    v\ge\psi\ \qe\text{ in }\Omega\}\neq \emptyset.
\end{equation}
Then $\K_{f,\psi}$ is a closed convex subset of
$H^s(\mathbb R^n)$.
\end{lemma}

\begin{proof}
    First, recall that by Theorem~\ref{thm: quasicontinuous representation} every $H^s$ function has a (global) quasicontinuous representative and hence the condition $v\geq \psi$ q.e.~in $\Omega$ in \eqref{eq: admissible set} is well-defined.
    
    Convexity of $\K_{f,\psi}$ is immediate from the affine condition $v-f\in\Htil(\Omega)$ and the pointwise convexity of the inequality $v\ge\psi$. 
    
    To prove closedness, let $v_k\in\K_{f,\psi}$, $k\in\N$, and $v_k\to v$ in $H^s(\mathbb R^n)$ as $k\to\infty$. Since
$\Htil(\Omega)$ is a closed linear subspace, $v-f\in\Htil(\Omega)$. Using Lemma~\ref{lemma: Stability of quasicontinuous representatives}, up to extracting a subsequence, we may assume that $v_k\to v$ quasieverywhere. Moreover, we have $v_k\ge\psi$ for every $k$ and $v_k\to v$ outside a set of capacity zero (see Theorem~\ref{thm: properties of cap}). Hence, $v\ge\psi$ q.e. in
$\Omega$ and therefore $v\in\K_{f,\psi}$.

This completes the proof.
\end{proof}

\begin{theorem}[Well-posedness of the obstacle problem]
\label{thm:obstacle-wp}
Let $\Omega\subset\mathbb R^n$ be an open bounded set and $0<s<1$. Let
$q\in L^\infty(\Omega)$ satisfy $q\ge0$ a.e. in $\Omega$. Let
$f\in H^s(\mathbb R^n)$, $\psi$ be quasicontinuous in $\Omega$, and
assume that $\K_{f,\psi}\neq \emptyset$. Then there is a unique
$u\in\K_{f,\psi}$ such that
\begin{equation}
\label{eq: variational inequality}
  \Bq(u,v-u)\ge0
  \qquad \forall v\in\K_{f,\psi}.
\end{equation}
Equivalently, $u$ is the unique minimizer of
\[
  J(v)\vcentcolon =\frac12\Bq(v,v)\quad\text{over }\K_{f,\psi}.
\]
\end{theorem}

\begin{proof}
    Using Theorem~\ref{thm: quasicontinuous representation}, we identify $f\in H^s(\R^n)$ with its quasicontinuous representative and hence the new obstacle $\widetilde \psi \vcentcolon =\psi-f$ is quasicontinuous in $\Omega$. Then
\[
C_{\widetilde \psi}\vcentcolon =
  \{w\in\Htil(\Omega): w\ge \widetilde\psi\ \qe\text{ in }\Omega\}
\]
is well-defined and satisfies
\[
  \K_{f,\psi}=f+C_{\widetilde \psi}.
\]
Applying Lemma~\ref{lem:closed-convex} with exterior datum $0$ and obstacle
$\widetilde\psi$ shows that $C_{\widetilde \psi}$ is nonempty, closed, and
convex in the Hilbert space $\Htil(\Omega)$. For all $w\in C_{\widetilde \psi}$,
\[
  J(f+w)
  =
  \frac12\Bq(w,w)+\Bq(f,w)+\frac12\Bq(f,f).
\]
Thus, $v=f+w\in \K_{f,\psi}$ minimizes $J$ over $\K_{f,\psi}$ if and only if $w\in C_{\widetilde \psi}$ minimizes 
\[
    \widetilde J(w)\vcentcolon = \frac12\Bq(w,w)+\Bq(f,w)
\]
over $C_{\widetilde \psi}$. By continuity of $\mathcal{B}_q$ on $H^s(\R^n)$ and \eqref{eq: coercivity estimate}, there exist $c>0$ and $\alpha_f>0$ such that 
\begin{equation}
\label{eq: well-posedness coercivity estimate}
    \widetilde J(w)\geq c\|w\|_{H^s(\R^n)}^2-\alpha_f
\end{equation}
for all $w\in C_{\widetilde \psi}$. Moreover, $\widetilde J$ is (strictly) convex. Recall that weak and strong closures of convex sets are equal and that convex continuous functionals are weakly lower semicontinuous. Thus, \cite[Theorem~1.2]{Variational-Methods} together with Lemma~\ref{lem:closed-convex} and strict convexity yields that $\widetilde J$ has a (unique) minimizer $w\in C_{\widetilde \psi}$. So, the same is true for $J$ over $\K_{f,\psi}$, which we denote by $u$.

If $v\in\K_{f,\psi}$, then convexity of $\K_{f,\psi}$ yields that
$u+t(v-u)\in\K_{f,\psi}$ for $0\le t\le1$.  Differentiating
$J(u+t(v-u))$ at $t=0$ gives the variational inequality \eqref{eq: variational inequality}. Conversely, if
$u \in \K_{f,\psi}$ satisfies \eqref{eq: variational inequality}, then for any $v\in\K_{f,\psi}$,
\[
  J(v)-J(u)
  =
  \Bq(u,v-u)+\frac12\Bq(v-u,v-u)
  \ge0,
\]
because $\Bq(v-u,v-u)\ge0$. Thus $u$ is a
minimizer.

Although, uniqueness of the minimizer $u$ follows from strict convexity of the energy functional, we will give next a more explicit argument directly working with the variational inequality. If $u_1,u_2\in \K_{f,\psi}$ are two solutions, testing the inequality for $u_1$
with $u_2$ and for $u_2$ with $u_1$ gives
\[
  \Bq(u_1,u_2-u_1)+\Bq(u_2,u_1-u_2)\ge0.
\]
Since the left hand side is $-\Bq(u_1-u_2,u_1-u_2)$, we get $\Bq(u_1-u_2,u_1-u_2)=0$. Thus, $q\geq 0$, the fractional Poincar\'e inequality, and $u_1-u_2\in\widetilde H^s(\Omega)$ imply $u_1=u_2$.
\end{proof}

\subsection{Openness of the non-contact set}

Next, we study the openness of the non-contact set $D=\{u>\psi\}$ for suitable compatible smooth obstacles. Additionally, it will show that in this case the obstacle solution $u$ satisfies $L_qu=0$ in the distributional sense in $D$, which improves Lemma~\ref{lem:equation-noncontact}. To this end, we introduce the multiplier associated with the obstacle problem and recall an abstract Lewy--Stampacchia result.

\begin{lemma}[Obstacle multiplier]
\label{lem:multiplier}
Let $\Omega\subset\mathbb R^n$ be an open bounded set and $0<s<1$. Let
$q\in L^\infty(\Omega)$ satisfy $q\ge0$ a.e. in $\Omega$. Let
$f\in H^s(\mathbb R^n)$, let $\psi$ be quasicontinuous in $\Omega$, assume
that $\K_{f,\psi}\neq \emptyset$, and let $u$ be the corresponding obstacle
solution (Theorem~\ref{thm:obstacle-wp}). The distribution $\lambda\in (\widetilde H^s(\Omega))'$, defined by
\begin{equation}
\label{eq: obstacle multiplier}
  \langle\lambda,\varphi\rangle \vcentcolon =\Bq(u,\varphi),
  \quad \varphi\in\Htil(\Omega),
\end{equation}
has the following properties:
\begin{enumerate}[(i)]
    \item\label{item: multiplier positive} $\lambda$ is nonnegative, that is,
    \[
        \langle \lambda,\varphi\rangle\geq 0
        \quad\text{for all }\varphi\in\widetilde H^s(\Omega),\ 
        \varphi\ge0\ \qe\text{ in }\Omega.
    \]
    \item\label{item: multiplier support} Let $D\vcentcolon =\{u>\psi\}$ in
    the quasiopen sense. Then, up to redefining $u$ on a set of capacity zero,
    \[
        \langle\lambda,\varphi\rangle=0
        \qquad \forall \varphi\in\Htil(D).
    \]
    Thus $\lambda$ is supported on the contact set in the capacitary sense.
    \item\label{item: multiplier complementarity} If, in addition,
    $\psi\in H^s(\mathbb R^n)$ and $\psi-f\in\Htil(\Omega)$, then
    \[
        \langle\lambda,u-\psi\rangle=0.
    \]
\end{enumerate}
\end{lemma}

\begin{proof}
    First of all, note that the continuity properties of $\mathcal{B}_q$ on $H^s(\R^n)$ and $u\in H^s(\R^n)$ imply that $\lambda$ is well-defined and belongs to the dual space $\widetilde H^s(\Omega)'$.
    
    Next, let us observe that $u+\varphi\in\K_{f,\psi}$, whenever $\varphi\in\Htil(\Omega)$ satisfies $\varphi\ge0$ q.e. in $\Omega$. Thus, the variational inequality \eqref{eq: variational inequality} with $v= u+\varphi$ yields $\langle\lambda,\varphi\rangle=\Bq(u,\varphi)\ge 0$. Hence, $\lambda$ is nonnegative.

We next prove the support statement, which uses only quasicontinuity of the
obstacle and the obstacle inequality. From here on, we assume that $u$ is
changed at most on a set of capacity zero such that $u$ is
quasicontinuous and hence $D$ is quasiopen (see Corollary~\ref{cor: quasiopen}). Furthermore, for $m\in\N$, set
$D_m=\{u-\psi>1/m\}$, which are again quasiopen. Let
$\varphi\in\Htil(D_m)\cap L^\infty(\Omega)$ be nonzero. For
$0<\varepsilon<1/(m\|\varphi\|_{L^\infty (\Omega)})$, both
$v_{+,\eps}\vcentcolon = u+\varepsilon\varphi$ and $v_{-,\eps}\vcentcolon = u-\varepsilon\varphi$ belong to the admissible class $\K_{f,\psi}$. Indeed, \cite[Theorem~10.1.1]{Adams-Hedberg} implies $\varphi=0$ q.e.~in $\R^n\setminus D_m$. Thus, $u\pm \eps\varphi\geq u\geq \psi$ q.e.~in $\Omega\setminus D_m$,
\[
    u\pm \eps\varphi\geq \psi+(1/m-\eps\|\varphi\|_{L^{\infty}(\Omega)})\geq \psi\quad \text{q.e.~in }D_m,
\]
and $u\pm\eps\varphi=u=f$ q.e.~in $\R^n\setminus \Omega$, implying $u\pm \eps \varphi-f\in \widetilde H^s(\Omega)$. For the last part of the statement, we used $u-f\in \widetilde H^s(\Omega)$ and again \cite[Theorem~10.1.1]{Adams-Hedberg}. Thus, $v_{\pm,\eps}\in \K_{f,\psi}$. 

Therefore, using \eqref{eq: obstacle multiplier} with $\varphi\in \widetilde H^s(D_m)\subset \widetilde H^s(\Omega)$ and testing \eqref{eq: variational inequality}
with both competitors $v_{\pm,\eps}$ yield
\[
    \pm \langle\lambda,\varphi\rangle= \pm \mathcal{B}_q(u,\varphi)=\frac{1}{\eps}\mathcal{B}_q(u,v_{\pm,\eps}-u)\geq 0.
\]
Thus, $\langle\lambda,\varphi\rangle=0$.

For a general $\varphi\in\Htil(D)$, we first use the capacitary exhaustion property (Lemma~\ref{lemma: exhaustion property}) to approximate $\varphi$ by a sequence $\varphi_k \in \widetilde H^s(D_{m_k})\cap L^{\infty}(\R^n)$. Then the previous analysis implies $\langle\lambda,\varphi_k\rangle=0$ for each $k\in\N$. Finally, by continuity of $\lambda$ on $\Htil(\Omega)$, we may pass to the limit $k\to\infty$. Thus, $\langle\lambda,\varphi\rangle=0$.

Finally, assume that $\psi\in H^s(\mathbb R^n)$ and
$\psi-f\in\Htil(\Omega)$. Then
$u-\psi=(u-f)-(\psi-f)\in\Htil(\Omega)$. For $0<t<1$, the function
$v_t\vcentcolon = u-t(u-\psi)=(1-t)u+t\psi$ remains above $\psi$ q.e.~in $\Omega$ and has
exterior value $f$, since $v_t-f\in \widetilde H^s(\Omega)$. Hence, $v_t\in \K_{f,\psi}$. By \eqref{eq: obstacle multiplier} and \eqref{eq: variational inequality} with $v=v_t$, we get
\[
    \langle\lambda,u-\psi\rangle=\Bq(u,u-\psi)=-\frac{1}{t}\Bq(u,v_t-u)\ge 0.
\]
Since $u-\psi\ge 0$ q.e.~in $\Omega$ and
$\lambda\ge 0$, the opposite inequality is automatic, hence equality holds.
\end{proof}

As an immediate consequence, we obtain:
\begin{corollary}[Equation in the noncontact set]
\label{lem:equation-noncontact}
Let $\Omega\subset\mathbb R^n$ be an open bounded set and $0<s<1$. Let
$q\in L^\infty(\Omega)$ satisfy $q\ge0$ a.e. in $\Omega$. Let
$f\in H^s(\mathbb R^n)$, and let $\psi$ be quasicontinuous in $\Omega$.
Assume that $\K_{f,\psi}\neq\emptyset$, let $u$ be the corresponding obstacle
solution, and set $D=\{u>\psi\}$ in the quasiopen sense. Then
\[
  \Bq(u,\varphi)=0
  \qquad \forall \varphi\in\Htil(D).
\]
Equivalently,
\[
  ((-\Delta)^s+q)u=0
  \quad\text{in }D
\]
in the weak sense.
\end{corollary}

\begin{proof}
By Lemma~\ref{lem:multiplier}\ref{item: multiplier support}, the multiplier
$\lambda\in\Htil(\Omega)'$ defined by
\[
  \langle\lambda,\varphi\rangle=\Bq(u,\varphi),
  \qquad \varphi\in\Htil(\Omega),
\]
satisfies
\[
  \langle\lambda,\varphi\rangle=0
  \qquad \forall \varphi\in\Htil(D).
\]
This is precisely the asserted identity.
\end{proof}

In the sequel, we shall use an abstract Lewy--Stampacchia result for which we need to introduce some terminology (see \cite{schaeferbanach,mosco2006implicit}). 

An ordered set $(X,\leq)$ is called \emph{lattice}, if the elements $x\vee y\vcentcolon = \sup(x,y)$ and $x\wedge y\vcentcolon = \inf (x,y)$ exist for all $x,y\in X$. Moreover, it is called \emph{vector lattice}, when $X$ is a (real) vector space such that the lattice operators are compatible with the vector operators, that is,
\[
    x\leq y\quad \Rightarrow \lambda x+z\leq \lambda y+z,\quad \text{for all }z\in X,\,\lambda\geq 0.
\]
In particular, any $x\in X$ can be decomposed as
\[
    x=x_+-x_{-}\quad\text{with}\quad x_{\pm}\vcentcolon = (\pm x)\vee 0.
\]
Moreover, we may define the absolute value function $|\,\cdot\,|$ by
\[
    |x|\vcentcolon = x_+ + x_{-}.
\]

Let now $V$ be a Hilbert space, which carries the structure of a vector lattice. Throughout this paragraph $V'$ denotes the topological dual of $V$, ordered
by the dual positive cone $P'$:
\[
  \ell_1\le \ell_2
  \quad\Longleftrightarrow\quad \ell_2-\ell_1\in P'\vcentcolon = \{\ell \in V': \langle \ell,v\rangle\geq 0\text{ for all }v\in V,\ v\ge 0\}.
\]
When we write $\ell_1\vee\ell_2$ or $\ell_1\wedge\ell_2$ in $V'$, we mean the
supremum or infimum with respect to this order, whenever it exists. The \emph{order dual} $V^*$ is then the subspace generated by $P'$, meaning that $V^*=P'-P'$.

The Lewy--Stampacchia theorem below considers operators $A\colon V\to V'$ with the following properties:
\begin{enumerate}[(a)]
    \item\label{hemicontinuity} $A$ is \emph{hemicontinuous}, that is, $[0,1]\ni t\mapsto \langle A(u+tv),w\rangle \in\R$ is continuous for all $u,v,w\in V$ (cf., e.g., \cite[Definition~3.2]{mosco2006implicit}).
    \item\label{coercivity mapping} $A$ is \emph{coercive}, that is, 
    \[
        \frac{\langle Av,v\rangle}{\|v\|_V}\to \infty\quad\text{as }\|v\|_V\to \infty;
    \]
    see \cite{Variational-Methods} or \cite[Eq.~(3.17)]{mosco2006implicit}.
    \item\label{t-monotone} $A$ is \emph{T-monotone}, that is, 
    \begin{equation}
    \label{eq: t-monotone op}
        \langle Au-Av,(u-v)_+\rangle \geq 0\quad\text{for all }u,v\in V;
    \end{equation}
    see \cite[Eq.~(2.2)]{mosco2006implicit}. We say $A$ is \emph{strictly T-monotone}, if it is T-monotone and equality in \eqref{eq: t-monotone op} implies $u\leq v$.
\end{enumerate}
Moreover, we note that in our setting \cite[Lemma~2.1]{mosco2006implicit} implies that a (strictly) T-monotone operator is in particular (strictly) monotone, meaning that
\begin{equation}
\label{eq: monotone operator}
    \langle Au-Av,u-v\rangle  \geq 0\quad\text{for all }u, v\in V,
\end{equation}
while strictly monotone means that \eqref{eq: monotone operator} holds with strict inequality sign whenever $u\neq v$.

Using the above notation, we have the following result.

\begin{theorem}[Abstract Lewy--Stampacchia theorem]
\label{thm:mosco-ls}
Let $V$ be a real Hilbert space which is also a vector lattice, and let
$A:V\to V'$ be bounded, coercive, hemicontinuous, and strictly T-monotone.

Let $\chi\in V$, $F\in V'$, and let $w\in V$ solve
\begin{equation}
\label{eq: assumps for abstract lewy stampacchia thm}
  w\ge\chi,\qquad
  \langle Aw-F,z-w\rangle\ge0
  \qquad\text{for every }z\in V,\ z\ge\chi.
\end{equation}
Then
\begin{equation}
\label{eq: lower bound LS}
  F\le Aw.
\end{equation}
Moreover, if $\Lambda\in V'$ satisfies
\begin{equation}
\label{eq: assumps on Lambda}
  \Lambda\ge F,\qquad \Lambda\ge A\chi,
\end{equation}
then
\begin{equation}
\label{eq: upper bound LS}
  Aw\le \Lambda.
\end{equation}
Consequently, if the order supremum $F\vee A\chi$ exists in $V'$, then
\begin{equation}
\label{eq: two-sided bound LS}
  F\le Aw\le F\vee A\chi .
\end{equation}
\end{theorem}

\begin{proof}
We apply Mosco's abstract dual estimate for the upper-obstacle problem
\cite[Theorem~4.1]{mosco2006implicit}.  We use it only in
the special case $X=V$, so Mosco's sublattice and obstacle compatibility
assumptions are automatic.  Also, in the present paper $V'$ denotes the
topological dual of $V$, ordered by the dual positive cone $P'$.

Mosco's theorem, in the notation used here, says the following. Let
$B:V\to V'$ satisfy the same structural assumptions as $A$. Let
$\varphi\in V$, $G\in V'$, and let $y\in V$ solve
\begin{equation}
\label{eq: variational ineq for B G phi}
  y\le \varphi,\qquad
  \langle By-G,\eta-y\rangle\ge0
  \qquad\text{for all }\eta\in V,\ \eta\le\varphi.
\end{equation}
Then
\begin{equation}
\label{eq: concl 1}
  By\le G.
\end{equation}
Moreover, if $h\in V'$ satisfies
\begin{equation}
\label{eq: conds on h}
  h\le G,\qquad h\le B\varphi,
\end{equation}
then
\begin{equation}
\label{eq: concl 2}
  h\le By.
\end{equation}

We reduce the lower-obstacle problem to this upper-obstacle form.  Set
\[
  y \vcentcolon =-w,\qquad \varphi \vcentcolon =-\chi,\qquad G \vcentcolon =-F,
\]
and define $B\colon V\to V'$ by $B z\vcentcolon =-A(-z)$ for $z\in V$.

Then the first condition in \eqref{eq: assumps for abstract lewy stampacchia thm} implies $y\le\varphi$.  If $\eta\le\varphi$ and $z\vcentcolon =-\eta$, then $z\ge\chi$ and hence the variational inequality in \eqref{eq: assumps for abstract lewy stampacchia thm} ensures
\[
\begin{split}
  \langle By-G,\eta-y\rangle
  &=
  \langle -A(-y)+F,\eta-y\rangle  =
  \langle Aw-F,z-w\rangle
  \ge0.
\end{split}
\]
Thus $y$ solves Mosco's upper-obstacle variational inequality \eqref{eq: variational ineq for B G phi} for $B$, $G$,
and $\varphi$.

Moreover, the operator $B$ satisfies the same structural assumptions as $A$. Boundedness
and hemicontinuity are immediate. Coercivity follows from
\[
  \langle By,y\rangle
  =
  \langle -A(-y),y\rangle
  =
  \langle A(-y),-y\rangle .
\]
Strict $T$-monotonicity is inherited as well.  Indeed, if
$(y_1-y_2)^+\ne0$, then, with $u=-y_2$ and $v=-y_1$,
\[
\begin{split}
  \langle By_1-By_2,(y_1-y_2)^+\rangle
  &=
  \langle A(-y_2)-A(-y_1),(y_1-y_2)^+\rangle \\
  &=
  \langle Au-Av,(u-v)^+\rangle
  >0.
\end{split}
\]
So, we can apply \cite[Theorem~4.1]{mosco2006implicit}. First, \eqref{eq: concl 1} gives $By\le G=-F$. Since $By=-Aw$, this yields
\[
  F\le Aw.
\]

Assume now that $\Lambda\in V'$ is a common upper bound of $F$ and $A\chi$,
that is, \eqref{eq: assumps on Lambda} holds.
Hence
\[
  -\Lambda\le -F=G,\qquad -\Lambda\le -A\chi=B\varphi.
\]
Thus $h\vcentcolon =-\Lambda$ satisfies \eqref{eq: conds on h} and by \eqref{eq: concl 2} we get $-\Lambda\le By$. Using again $By=-Aw$, we obtain
\[
 Aw\le \Lambda.
\]
Combining the two bounds gives
\[
  F\le Aw\le \Lambda
  \qquad\text{in }V'.
\]
The estimate \eqref{eq: two-sided bound LS} follows by applying this with $\Lambda\vcentcolon = F\vee A\chi$.
\end{proof}

\begin{remark}
    The proof above uses Mosco's theorem \cite[Theorem~4.1]{mosco2006implicit} in
its form with an arbitrary lower bound $h\in V'$ satisfying
$h\le G$ and $h\le B\varphi$. Therefore no membership in the order dual $V^*$
is needed. The assumption that $F\vee A\chi$ exists in our ordered dual space
only provides the particular lower bound $h=-(F\vee A\chi)$ after the sign
change.

The more recent theorem of Rodrigues--Teymurazyan
\cite[Theorem~4.2]{Rodrigues-Teymurazyan-LS} gives a two-obstacle extension
of Mosco's estimate in the same abstract $T$-monotone framework.  In
particular, their Proposition~4.4 recovers the one-obstacle
Lewy--Stampacchia inequalities.  For the present one-obstacle theorem,
Mosco's original result is sufficient.

Moreover, fractional and nonlocal variants of Lewy--Stampacchia inequalities are also
discussed in \cite{Kow-Kimura-LS,Lo-Rodrigues-LS}.
\end{remark}

Our next result specializes to the fractional Schrödinger operator with compatible smooth data.

\begin{lemma}[Lewy--Stampacchia bound for compatible smooth obstacles]
\label{lem:Lewy--Stampacchia}
Let $\Omega\subset\mathbb R^n$ be an open bounded set and $0<s<1$. Let
$q\in L^\infty(\Omega)$ satisfy $q\ge0$ a.e. in $\Omega$.  Let
$f\in H^s(\mathbb R^n)$, let $\psi$ be quasicontinuous in $\Omega$, assume
that $\K_{f,\psi}$ is nonempty, and let $u$ be the corresponding obstacle
solution (Theorem~\ref{thm:obstacle-wp}).  Assume that there is a function
$\Psi\in H^s(\mathbb R^n)\cap L^\infty(\mathbb R^n)$ such that
\begin{equation}
\label{eq: representation of f and psi by common function}
  \Psi-f\in\Htil(\Omega),\qquad \Psi=\psi\quad\qe\text{ in }\Omega,
\end{equation}
and such that both distributions
\begin{equation}
\label{eq: schrödinger source }
  L_q f:=(-\Delta)^s f+qf,\qquad
  L_q\Psi:=(-\Delta)^s\Psi+q\Psi
\end{equation}
are represented in $\Omega$ by functions in $L^\infty(\Omega)$.  Then the obstacle
multiplier
\[
  \lambda=L_q u\in H^{-s}(\Omega),
\]
defined in Lemma~\ref{lem:multiplier}, satisfies
\begin{equation}
\label{eq: bounds for obstacle multiplier}
  0\le \lambda\le (L_q\Psi)_+ .
\end{equation}
In particular $\lambda$ is represented by an $L^\infty(\Omega)$ function.
\end{lemma}

\begin{proof}
The lower bound $0\le \lambda$ is already Lemma~\ref{lem:multiplier}\ref{item: multiplier positive}. We now prove the upper bound in \eqref{eq: bounds for obstacle multiplier} from Theorem~\ref{thm:mosco-ls}. We will apply it with the choices:
\[
    V\vcentcolon =\Htil(\Omega)\quad \text{and}\quad A\vcentcolon = L_q.
\]
Moreover, to match the notation in that result, let us set $w\vcentcolon =u-f$ and $\chi\vcentcolon =\Psi-f$. Then $w,\chi\in\Htil(\Omega)$ and the
obstacle problem \eqref{eq: variational inequality} is equivalently 
\[
  w\in C_\chi\quad\text{with}\quad\Bq(w,z-w)\ge -\Bq(f,z-w)\quad\text{for all }z\in C_\chi,
\]
where 
\[
    C_\chi \vcentcolon = \{z\in\Htil(\Omega):z\ge\chi\ \qe\text{ in }\Omega\}.
\]
Thus the data in the homogeneous problem are the
functional
\[
    F\vcentcolon =-\Bq(f,\,\cdot\,)=\langle-L_qf,\,\cdot\,\rangle
\]
and the obstacle is $\chi$. 

On $V=\Htil(\Omega)$ we use the capacitary order
\[
  v\le w
  \quad\Longleftrightarrow\quad
  v\le w\quad\text{q.e. in }\Omega,
\]
where all pointwise statements are understood for quasicontinuous
representatives. This is well-defined by uniqueness of quasicontinuous
representatives. See Theorem~\ref{thm: quasicontinuous representation} and Remark~\ref{rem: Uniqueness of quasicontinuous representations}. With
this order,
\[
  v\vee w=w+(v-w)_+,\qquad v\wedge w=v-(v-w)_+,
\]
and Lemma~\ref{lem:lattice}, \ref{item: lipschitz map inhomogeneous} together with \cite[Theorem~10.1.1]{Adams-Hedberg} shows that these functions belong to
$\Htil(\Omega)$. Thus $\Htil(\Omega)$ is a vector lattice.
The dual space $\Htil(\Omega)'$ is ordered by the dual positive cone:
\[
  \ell_1\le \ell_2
  \quad\Longleftrightarrow\quad
  \langle \ell_2-\ell_1,\varphi\rangle\ge0
  \quad\text{for all }\varphi\in\Htil(\Omega),\ \varphi\ge0\ \text{q.e.}
\]

Next, we verify the conditions for $A=L_q$, which is represented by the bilinear form $\Bq$. Since $q\ge0$, the bilinear form
$\Bq$ is symmetric, continuous, and coercive on $\Htil(\Omega)$; see Section~\ref{subsec: functional-analytic setup}. Thus, boundedness, hemicontinuity, and coercivity of $L_q$ are immediate.

So, it remains to check that $L_q$ is strictly T-monotone. Indeed, for all
$y_1,y_2\in\Htil(\Omega)$,
\[
\begin{split}
    \langle L_qy_1-L_qy_2,(y_1-y_2)_+\rangle
  &=
  \Bq(y_1-y_2,(y_1-y_2)_+)\\
  &\ge
  \Bq((y_1-y_2)_+,(y_1-y_2)_+)\ge0.
\end{split}
\]
If $(y_1-y_2)_+\ne0$, the last quantity is strictly positive by coercivity \eqref{eq: coercivity estimate}.

By assumption, the functional $F=-L_q f$ and the element
$L_q\chi=L_q(\Psi-f)$ are represented by $L^\infty(\Omega)$ functions. Set
\[
  \Lambda \vcentcolon = (L_q\Psi)_+ - L_qf .
\]
Then $\Lambda\in L^\infty(\Omega)$, hence it defines an element of
$\Htil(\Omega)'$.  Moreover, in the dual cone order,
\[
  \Lambda\ge -L_qf=F,
\]
because $(L_q\Psi)_+\ge0$, and
\[
  \Lambda\ge L_q\Psi-L_qf=L_q(\Psi-f)=L_q\chi,
\]
because $(L_q\Psi)_+\ge L_q\Psi$. Thus $\Lambda$ is an admissible upper
bound for both $F$ and $L_q\chi$ in Theorem~\ref{thm:mosco-ls} (see \eqref{eq: assumps on Lambda}).  Applying
that theorem gives
\[
  F\le L_qw\le \Lambda .
\]
Adding $L_qf$ and using $u=f+w$, we obtain
\[
  0\le L_qu\le (L_q\Psi)_+ .
\] 

Finally, we note that a positive distribution dominated
by an $L^\infty$ function is represented by an $L^\infty$ density. First,
for nonnegative $\varphi\in C_c^\infty(\Omega)$, we have
\begin{equation}
\label{eq: two-sided bound multiplier obstacle}
  0\le \langle L_q u,\varphi\rangle
  \le \int_\Omega (L_q\Psi)_+\varphi\,dx.
\end{equation}
Hence, by Riesz's representation theorem, \cite[Section~1.8]{evans-mt}, there exists a (finite) Radon measure $\nu$ such that
\[
    \langle L_q u,\varphi\rangle=\int_\Omega \varphi\,d\nu.
\]
By mollification, we conclude that $\nu\ll \mu\vcentcolon = (L_q\Psi)_+\,dx$. Indeed, first 
\[
    \int_\Omega \varphi\,d\nu\leq \int_\Omega (L_q\Psi)_+\varphi\,dx
\]
holds for any $\varphi\in C_c(\Omega)$, which implies
\[
    \nu(C)\leq \int_C (L_q\Psi)_+\,dx
\]
for any measurable $C$ compactly contained in $\Omega$. Hence, using an exhaustion by compact sets of $\Omega$, denoted by $(K_n)_{n\in\N}$, we find that
\begin{equation}
\label{eq: estimate between nu mu}
\begin{split}
    \nu(D)&=\lim_{n\to\infty}\nu(D\cap K_n)\\
    &\leq \lim_{n\to\infty}\int_{D\cap K_n}(L_q\Psi)_+\,dx\\
    &= \int_{D}(L_q\Psi)_+\,dx
\end{split}
\end{equation}
for any measurable set $D\subset \Omega$. In the first and last equality we used monotone convergence. 

In particular, $\nu \ll \mu\vcentcolon = (L_q\Psi)_+\,dx$ and so the Radon--Nikodym theorem, \cite[Section~1.6]{evans-mt}, yields
\[
    \nu(C)=\int_{C}\rho\,dx
\]
for all measurable sets $C\subset \Omega$. Indeed, $\rho=(D_\mu \nu)(L_q\Psi)_+$ in the notation of \cite{evans-mt}. Now, \eqref{eq: estimate between nu mu} demonstrates that $\rho\in L^1(\Omega)$ and so we deduce from Lebesgue's differentiation theorem that
\[
    \rho(x)\leq (L_q\Psi)_+(x)
\]
for a.e.~$x\in\Omega$. Especially, $\rho\in L^{\infty}(\Omega)$. This concludes the proof.
\end{proof}

\begin{lemma}[Local continuity for bounded fractional Poisson solutions]
\label{lem:local-continuity}
Let $\Omega\subset\mathbb R^n$ be open, let $0<s<1$, let
$u\in H^s(\mathbb R^n)\cap L^\infty(\mathbb R^n)$, and let
$g\in L^\infty_{\rm loc}(\Omega)$.  Assume that
\[
  (-\Delta)^s u=g\quad\text{in }\mathcal D'(\Omega).
\]
Then $u$ has a continuous representative in $\Omega$.
\end{lemma}

\begin{proof}
It is enough to prove continuity in an arbitrary ball
$B_R(x_0)\Subset\Omega$.  Let $w\in H^s(\mathbb R^n)$ be the weak solution of
\[
    \begin{cases}
         (-\Delta)^s w=g\quad&\text{in }B_R(x_0),\\
         w=0\quad&\text{in }\mathbb R^n\setminus B_R(x_0).
    \end{cases}
\]
Since $g\in L^\infty(B_R(x_0))$ and the ball is smooth, the Dirichlet
regularity theorem of Ros-Oton and Serra
\cite[Proposition 1.1]{ROS-Dirichlet} gives
$w\in C^s(\mathbb R^n)$. 

Set $h=u-w$.  Then $h\in H^s(\mathbb R^n)\cap L^\infty(\mathbb R^n)$ and
\[
  (-\Delta)^s h=0\quad\text{in }\mathcal D'(B_R(x_0)).
\]
Let $B_{2r}(x_1)\Subset B_R(x_0)$ and let $\rho_\varepsilon$ be a standard
mollifier with sufficiently small $\eps>0$.
Set $h_\varepsilon=\rho_\varepsilon*h$. Then $h_\varepsilon$ is smooth and
bounded, with
\[
  \|h_\varepsilon\|_{L^\infty(\mathbb R^n)}
  \le \|h\|_{L^\infty(\mathbb R^n)}.
\]
Moreover,
\[
  (-\Delta)^s h_\varepsilon=0
  \quad\text{pointwise in }B_{2r}(x_1).
\]
Indeed, $(-\Delta)^s$ commutes with convolution in the distributional sense,
and the mollification of a test function supported in $B_{2r}(x_1)$ remains
supported in $B_R(x_0)$ for such $\varepsilon$.

By Silvestre's \cite[Theorem~5.1]{Silvestre-Holder}, together
with the verification of the fractional Laplacian kernel in
\cite[Section~3.1]{Silvestre-Holder}, there exist $\alpha\in(0,1)$ and
$C>0$, depending only on $n$ and $s$, such that
\[
  \|h_\varepsilon\|_{C^\alpha(B_r(x_1))}
  \le
  C r^{-\alpha}\|h\|_{L^\infty(\mathbb R^n)}.
\]
The estimate is uniform in $\varepsilon$. Since
$h_\varepsilon\to h$ in $L^2(B_R(x_0))$, the
Arzelà--Ascoli theorem and uniqueness of the $L^2$ limit show
that $h$ has a Hölder continuous representative in $B_r(x_1)$.

Since $B_{2r}(x_1)\Subset B_R(x_0)$ was arbitrary, $h$ is locally Hölder
continuous in $B_R(x_0)$.  As $w\in C^s(\mathbb R^n)$, it follows that
$u=w+h$ is continuous in $B_R(x_0)$.
\end{proof}

\begin{proposition}[A sufficient condition for openness]
\label{prop:open-noncontact}
Let $\Omega\subset\mathbb R^n$ be an open bounded set and $0<s<1$. Let
$q\in L^\infty(\Omega)$ satisfy $q\ge0$ a.e. in $\Omega$. Let
\[
  f\in H^s(\mathbb R^n)\cap L^\infty(\mathbb R^n),\qquad
  \psi\in C(\Omega)\cap L^\infty(\Omega).
\]
Assume that there is a compatible smooth representative $\Psi\in H^s(\R^n)\cap L^\infty (\R^n)$ of the data $(f,\psi)$, that is, the conditions \eqref{eq: representation of f and psi by common function}--\eqref{eq: schrödinger source } hold. Furthermore, suppose that $\K_{f,\psi}$ is nonempty and let $u$ be the corresponding
obstacle solution.

Then
the obstacle solution $u$ has a continuous representative in $\Omega$.
Consequently
\[
  D=\{x\in\Omega:\ u(x)>\psi(x)\}
\]
is an open subset of $\Omega$.  In this case the equation
$((-\Delta)^s+q)u=0$ holds in $D$ in the usual distributional sense.
\end{proposition}

\begin{proof}
We first establish the $L^\infty$ bound of the solution $u$, which is needed to invoke the regularity result Lemma~\ref{lem:local-continuity}.

Let $ M\vcentcolon = \max\{\|f\|_{L^\infty(\mathbb R^n)},\|\psi\|_{L^\infty(\Omega)}\}$. As usual, we use quasicontinuous representatives. Since $f\in H^s(\mathbb R^n)\cap L^\infty(\mathbb R^n)$, the bound
$|f|\le \|f\|_{L^\infty(\mathbb R^n)}$ holds q.e. in $\R^n$. Indeed, let
$\rho_\varepsilon$ be a standard mollifier. Then
$\rho_\varepsilon*f\to f$ strongly in $H^s(\mathbb R^n)$ and
\[
  |\rho_\varepsilon*f|\le \|f\|_{L^\infty(\mathbb R^n)}
  \qquad\text{everywhere in }\mathbb R^n.
\]
By Lemma~\ref{lemma: Stability of quasicontinuous representatives}, there is
a sequence $\varepsilon_j\downarrow0$ such that
$\rho_{\varepsilon_j}*f\to f$ q.e. in $\R^n$.  Passing to the limit along this sequence gives
\[
  |f|\le \|f\|_{L^\infty(\mathbb R^n)}\qquad\text{q.e. in }\mathbb R^n.
\]
The truncation $w=(u-M)_+$ belongs to
$H^s(\mathbb R^n)$ by Lemma \ref{lem:lattice}, because $M\ge0$.  Moreover
$u-f\in\Htil(\Omega)$ and \cite[Theorem~10.1.1]{Adams-Hedberg} give $u=f$ q.e. on $\mathbb R^n\setminus\Omega$, and
therefore $w=0$ q.e. on $\mathbb R^n\setminus\Omega$. Now, \cite[Theorem~10.1.1]{Adams-Hedberg} yields $w\in\Htil(\Omega)$.  The function
$u-w=\min\{u,M\}$ is an admissible function for the obstacle problem because $M\ge\psi$ q.e.~in $\Omega$. Testing the variational
inequality \eqref{eq: variational inequality} with the competitor $u-w$ gives
\[
  \Bq(u,-w)\ge0,
  \qquad\text{hence}\qquad
  \Bq(u,w)\le0.
\]
On the set where $w>0$ one has $u=M+w$, while $w=0$ elsewhere.  Hence,
using $q\ge0$ and $M\ge0$,
\begin{equation}
\label{eq: potential lower bound pos}
  \int_\Omega q u w\,dx
  =
  \int_{\{w>0\}} q(M+w)w\,dx
  \ge
  \int_\Omega q w^2\,dx .
\end{equation}
Together with 
\begin{equation}
\label{eq: fractional energy estimate below}
\begin{split}
    &(u(x)-u(y))(w(x)-w(y))=((u-M)(x)-(u-M)(y))(w(x)-w(y))\\
    &=|w(x)-w(y)|^2-((u-M)_{-}(x)-(u-M)_{-}(y))(w(x)-w(y))\\
    &=|w(x)-w(y)|^2+(u-M)_{-}(x)w(y)+(u-M)_{-}(y)w(x)\\
    &\geq |w(x)-w(y)|^2,
\end{split}
\end{equation}
this yields
\[
  \Bq(w,w)\le \Bq(u,w)\le0.
\]
Coercivity gives $w=0$, and hence
\[
    u\leq M \quad\text{q.e.~in }\R^n.
\]

For the lower bound set $z=(-M-u)_+$.  Then
$z\in H^s(\mathbb R^n)$ by Lemma \ref{lem:lattice}, because $-M\le0$.
Moreover $z=0$ q.e. on $\mathbb R^n\setminus\Omega$, since $u=f\ge -M$ there
q.e.; hence $z\in\Htil(\Omega)$.  The competitor
$u+z=\max\{u,-M\}$ is admissible since it is q.e. above $u$.  The
variational inequality gives $\Bq(u,z)\ge0$.  On $\{z>0\}$ one has
$u=-M-z$, while $z=0$ elsewhere.  Thus, using again $q\ge0$ and $M\ge0$,
\[
  \int_\Omega q u z\,dx
  =
  -\int_{\{z>0\}} q(M+z)z\,dx
  \le
  -\int_\Omega q z^2\,dx .
\]
Together with 
\[
\begin{split}
    (u(x)-u(y))(z(x)-z(y))&=-((-M-u)(x)-(-M-u)(y))(z(x)-z(y))\\
    &\leq -|z(x)-z(y)|^2,
\end{split}
\]
see \eqref{eq: fractional energy estimate below}, this gives
\begin{equation}
\label{eq: bilinear form estimate lower}
  \Bq(u,z)\le -\Bq(z,z)\le0.
\end{equation}
Thus $\Bq(z,z)=0$, and coercivity gives $z=0$.  Hence
$u\in L^\infty(\mathbb R^n)$ with
\[
    |u|\leq M\quad\text{q.e.~in }\R^n.
\]

Now by Lemma \ref{lem:Lewy--Stampacchia}, the obstacle multiplier
$\lambda=L_q u$ belongs to $L^\infty(\Omega)$.  Hence
\[
  (-\Delta)^s u=\lambda-qu
\]
is represented by an $L^\infty(\Omega)$ function.  Thus Lemma~\ref{lem:local-continuity} gives a continuous representative of $u$ in
$\Omega$.

Now $u-\psi$ is continuous, and therefore its positivity set is open.  If
$\varphi\in C_c^\infty(D)$, compactness of $\supp\varphi$ and continuity give
\[
  u-\psi\ge \delta>0
  \qquad\text{on }\supp\varphi
\]
for some $\delta$.  For all sufficiently small $|\varepsilon|$, the
competitors $u\pm\varepsilon\varphi$ remain admissible.  Testing the
variational inequality \eqref{eq: variational inequality} with both signs yields $\Bq(u,\varphi)=0$, which is
exactly the distributional form of $L_qu \vcentcolon =((-\Delta)^s+q)u=0$ in $D$.
\end{proof}

\begin{remark}[Regularity input from Lewy--Stampacchia]
\label{rem:book-openness}
Proposition \ref{prop:open-noncontact} does not require continuity of $u$ as
an independent hypothesis.  The assumption $q\ge0$, together with
$f\in L^\infty$ and $\psi\in L^\infty$, gives the bound
$u\in L^\infty$ by comparison.  The compatible smooth representative
$\Psi$ then lets us apply the Lewy--Stampacchia inequality, which upgrades
the multiplier $L_q u$ from a nonnegative distribution to an $L^\infty$
function.
Thus $(-\Delta)^s u=L_q u-qu$ has a bounded right-hand side in $\Omega$, and
linear fractional regularity gives continuity. 
\end{remark}

\begin{remark}[Why smooth compatibility is imposed]
\label{rem:cz-obstacle}
The compatibility condition $\Psi-f\in\Htil(\Omega)$ is not cosmetic.  The
Lewy--Stampacchia upper bound compares the multiplier $\lambda$ with $L_q\Psi$, where
$\Psi$ is the obstacle in the same affine exterior class as the admissible
functions.  A global smooth extension of $\psi$ supported in a neighborhood
of $\Omega$ is therefore sufficient only for exterior data that agree with
that extension near $\partial\Omega$ in the $H^s$ trace sense.  A simple
uniform sufficient condition for the scaled exterior data used below is
$\psi\in C_c^\infty(\Omega)$ and $f_0\in C_c^\infty(W_1)$ with
$W_1\Subset\Omega_e\vcentcolon = \R^n\setminus \overline{\Omega}$; then the piecewise function equal to $\psi$ in
$\Omega$ and to $t f_0$ in $\Omega_e$ is smooth across a neighborhood of
$\partial\Omega$ and is an admissible compatible representative.
\end{remark}

\subsection{Dirichlet to Neumann map}

\begin{lemma}[The exterior DN map]
\label{lem:dn-map}
Let $\Omega\subset\mathbb R^n$ be an open bounded set and $0<s<1$. Let
$q\in L^\infty(\Omega)$ satisfy $q\ge0$ a.e. in $\Omega$. Let
$f\in H^s(\mathbb R^n)$ and let $\psi$ be quasicontinuous in $\Omega$.
Assume that $\K_{f,\psi}$ is nonempty, let $u$ be the corresponding obstacle
solution, and let $W\subset\Omega_e$ be open. Then
the formula
\[
  \Lambda_{q,\psi}(f)|_W \vcentcolon =(-\Delta)^s u|_W
\]
defines an element of $\mathcal D'(W)$.  Equivalently, for
$\phi\in C_c^\infty(W)$,
\[
  \langle \Lambda_{q,\psi}(f),\phi\rangle
  =
  \langle (-\Delta)^s u,\phi\rangle
  =
  \mathcal E_s(u,\phi).
\]
If two representatives of $f|_{\R^n\setminus\Omega}$ differ by an element of $\Htil(\Omega)$, then the
corresponding admissible sets and obstacle solution are the same.  Thus the
map depends only on the usual exterior Dirichlet class
$H^s(\mathbb R^n)/\Htil(\Omega)$.
\end{lemma}

\begin{proof}
For $u\in H^s(\mathbb R^n)$, the fractional Laplacian is defined as an
element of $H^{-s}(\mathbb R^n)$ by
\[
  \langle(-\Delta)^s u,\phi\rangle=\mathcal E_s(u,\phi),
  \qquad \phi\in H^s(\mathbb R^n).
\]
Since $C_c^\infty(W)\subset H^s(\mathbb R^n)$, restriction to $W$ is a
well-defined distribution.  If $f_1-f_2\in\Htil(\Omega)$, then the affine conditions
$v-f_1\in\Htil(\Omega)$ and $v-f_2\in\Htil(\Omega)$ are equivalent.  The
constraint $v\ge\psi$ is unchanged, and uniqueness in Theorem
\ref{thm:obstacle-wp} gives the same state.
\end{proof}

\section{Unique continuation theorems}
\label{sec: UCP}

We use two standard unique continuation facts for the fractional Laplacian.
The first is the UCP of Ghosh, Salo, and Uhlmann \cite{GSU-fractional}.
The second is the measurable UCP of the fractional Schrödinger operator used in the
single-measurement reconstruction of Ghosh, Rüland, Salo, and Uhlmann \cite{GRSU-single}.

\begin{theorem}[UCP, {\cite[Theorem 1.2]{GSU-fractional}}]
\label{thm:antilocality}
Let $n\ge 1$ and $0<s<1$.  If $v\in H^{-r}(\mathbb R^n)$ for some
$r\in\mathbb R$ satisfies
\[
  (-\Delta)^s v=v=0
  \quad\text{in a nonempty open set }W\subset\R^n,
\]
then $v\equiv0$ in $\R^n$.
\end{theorem}

\begin{theorem}[Measurable UCP, {\cite[Theorem 3]{GRSU-single}}]
\label{thm:meas-ucp}
Let $\Omega\subset\mathbb R^n$, $n\ge1$, be an open set,
$s\in[1/4,1)$, and let $q\in L^\infty(\Omega)$.  If $v\in H^s(\mathbb R^n)$
satisfies
\[
  ((-\Delta)^s+q)v=0\quad\text{in }\Omega
\]
and $v$ vanishes on a subset of $\Omega$ of positive Lebesgue measure, then
$v\equiv0$ in $\mathbb R^n$.
\end{theorem}

\begin{remark}
Theorem \ref{thm:meas-ucp} is the $L^\infty$ version needed below. Since it
is stated for an arbitrary open set, we can apply it with
$\Omega$ replaced by an open noncontact sets $D\subset\Omega$ and with
$q$ replaced by $q|_D$. For $s\in(0,1/4)$ this $L^\infty$ measurable UCP is not proved in \cite{GRSU-single}. The obstruction in their
argument is the absorption of the boundary term in the Carleman estimate;
see \cite[Remark 5.6]{GRSU-single}. In the range $s\in(0,1/4)$ one may
instead either use the continuous-potential argument in the proof of
\cite[Theorem 1]{GRSU-single}, which only requires weak unique continuation,
or assume a potential class for which measurable UCP is
known. A concrete such class is
$q\in C^1(\Omega)$, for which the measurable unique continuation property is
proved by Fall and Felli \cite{Fall-Felli-UCP}; \cite[Remark 5.6]{GRSU-single}
also notes that the Carleman approach would only use $C^1$ regularity in the
radial directions.
\end{remark}

\section{Single-state recovery}
\label{sec: Single state}

In this section, we investigate injectivity properties of the DN data from a single measurement. 

\begin{theorem}[Recovery from one obstacle state]
\label{thm:one-state}
Let $\Omega\subset\mathbb R^n$ be an open bounded set, $0<s<1$, and let
$q_1,q_2\in L^\infty(\Omega)$ satisfy $q_j\ge 0$ a.e. in $\Omega$,
$j=1,2$.  Let $f\in H^s(\mathbb R^n)$ and let $\psi$ be quasicontinuous in
$\Omega$.  Assume that $\K_{f,\psi}$ is nonempty, and let
$u_j$ be the obstacle solution for exterior datum $f$, obstacle $\psi$, and
potential $q_j$.  Assume
\[
  \Lambda_{q_1,\psi}(f)|_W=\Lambda_{q_2,\psi}(f)|_W
\]
on a nonempty open set $W\subset\Omega_e$.  Then
\[
  u_1=u_2=:u \quad\text{in }\R^n.
\]
Consequently the contact and noncontact sets agree.  In the common noncontact
set
\[
  D=\{u>\psi\}
\]
one has
\[
  (q_1-q_2)u=0
  \quad\text{in the weak }\Htil(D)\text{ sense}.
\]
Assume additionally that, for one of the two potentials $q_j$, the explicit
hypotheses of Proposition \ref{prop:open-noncontact} hold; namely
$f\in H^s(\mathbb R^n)\cap L^\infty(\mathbb R^n)$ and
$\psi\in C(\Omega)\cap L^\infty(\Omega)$ with a compatible smooth
representative $\Psi\in H^s(\R^n)\cap L^{\infty}(\R^n)$ satisfying \eqref{eq: representation of f and psi by common function}--\eqref{eq: schrödinger source } with $q=q_j$.  Then
$D$ is open and the preceding identity holds in $\mathcal D'(D)$. If $s\geq 1/4$ and at least one of $f\not\equiv 0$ and $\psi \not\leq 0$ holds, then
\[
  q_1=q_2\quad\text{a.e. in }D.
\]
If $0<s<1/4$, at least one of $f\not\equiv0$ and $\psi \not\leq 0$ holds, and $q_1,q_2$ have continuous representatives in $D$, then
one has $q_1=q_2$ throughout $D$.
\end{theorem}

\begin{proof}
Set $v=u_1-u_2$.  Since the exterior data are the same, $v=0$ in
$\Omega_e$, hence in $W$.  Equality of the measured DN data gives
$(-\Delta)^s v=0$ in $W$.  Theorem \ref{thm:antilocality} implies
$v\equiv0$ in $\R^n$.

In $D=\{u>\psi\}$, Corollary~\ref{lem:equation-noncontact} applied to both
obstacle problems gives
\begin{equation}
\label{eq: fractional schroedinger single}
    ((-\Delta)^s+q_j)u=0
  \quad\text{in }D,\qquad j=1,2.
\end{equation}
Subtracting yields $(q_1-q_2)u=0$ in $D$.

From now on, after relabeling the potentials if necessary, Proposition~\ref{prop:open-noncontact} shows that $D$ is open and \eqref{eq: fractional schroedinger single} holds in the usual distributional sense for $j=1$. Assume first $s\geq 1/4$. We exclude positive-measure zero sets of a nontrivial state.
If $E\subset D$ has positive measure and $u=0$ a.e. on $E$, then $u$
satisfies
\[
  ((-\Delta)^s+q_1|_D)u=0\quad\text{in }D
\]
and vanishes on a positive-measure subset of $D$.  The measurable UCP, Theorem
\ref{thm:meas-ucp} applied with $D$ in place of $\Omega$, then gives $u\equiv0$ in $\R^n$, contrary to one of the hypotheses $f\not\equiv0$ or $\psi \not \leq 0$.
Hence $u\ne0$ a.e.~in $D$, and $(q_1-q_2)u=0$ implies $q_1=q_2$ a.e.~in
$D$.

If instead $s<1/4$ and $q_1,q_2$ are continuous in $D$, then suppose that
$(q_1-q_2)(x_0)\ne0$ for some $x_0\in D$. Then there is a nonempty open ball
$B\Subset D$ on which $q_1-q_2$ does not vanish.  Since
$(q_1-q_2)u=0$ in $D$, we have $u=0$ a.e. in $B$.  In $B$ the equation
$((-\Delta)^s+q_1)u=0$ then gives $(-\Delta)^s u=0$ in the distributional
sense.  Theorem \ref{thm:antilocality} applied in $B$ gives
$u\equiv0$ in $\R^n$, contradicting the nontriviality of $f$ or $\psi \not \leq 0$. Thus
$q_1=q_2$ throughout $D$.
\end{proof}

\begin{corollary}[Recovery on a union of exposed sets]
\label{cor:union}
Let $\Omega\subset\mathbb R^n$ be an open bounded set, $0<s<1$, and let
$q_1,q_2\in L^\infty(\Omega)$ satisfy $q_j\ge 0$ a.e. in $\Omega$,
$j=1,2$. Let $\psi$ be quasicontinuous in $\Omega$, suppose
$W\subset\Omega_e$ is a nonempty open set, and let $\mathcal F$ be a family of
exterior data $f\in H^s(\mathbb R^n)$ such that $\K_{f,\psi}$ is nonempty
for every $f\in\mathcal F$. 

Furthermore, suppose that $\mathcal{F}\subset H^s(\R^n)\setminus\{0\}$ or $\psi\not\leq 0$, the openness hypotheses in Proposition
\ref{prop:open-noncontact} hold for each $f\in\mathcal F$ for one of the two potentials, and either $s\geq 1/4$ or $q_1,q_2$ have continuous representatives in $\Omega$.

If the DN maps satisfy
\[
  \Lambda_{q_1,\psi}(f)|_W=\Lambda_{q_2,\psi}(f)|_W
  \qquad \forall f\in\mathcal F,
\]
then the sets $\{u_f>\psi\}$ are open and
\[
  q_1=q_2
  \quad\text{a.e. in}\quad
  \bigcup_{f\in\mathcal F}\{u_f>\psi\}.
\]
\end{corollary}

\begin{proof}
Apply Theorem \ref{thm:one-state} to each $f\in\mathcal F$. This gives
$q_1=q_2$ a.e. in each open noncontact set
$D_f=\{u_f>\psi\}$.  The union
\[
  U=\bigcup_{f\in\mathcal F}D_f
\]
is open in $\mathbb R^n$.  Since Euclidean open sets are second countable and hence Lindelöf,
there is a countable subfamily $\{f_k\}_{k=1}^\infty\subset\mathcal F$ such
that
\[
  U=\bigcup_{k=1}^\infty D_{f_k}.
\]
The exceptional null sets on the countably many $D_{f_k}$ still have null
union, and hence $q_1=q_2$ a.e. in $U$.
\end{proof}

\begin{remark}[State reconstruction]
\label{rem:reconstruction}
The proof is constructive in the same sense as the single-measurement
fractional Calderón reconstruction.  Write $u=f+w$ with
$w\in\Htil(\Omega)$.  From the exterior measurement one knows
\[
  (-\Delta)^s w|_W
  =
  \Lambda_{q,\psi}(f)|_W-(-\Delta)^s f|_W.
\]
The operator $w\mapsto (-\Delta)^s w|_W$ is injective on $\Htil(\Omega)$ by
Theorem \ref{thm:antilocality}; the Tikhonov and spectral schemes in
\cite[Theorem 2 and Section 3]{GRSU-single} recover $w$, hence $u$.  After
this, in the noncontact set,
\[
  q=-\frac{(-\Delta)^s u}{u}
  \qquad\text{where }u\ne0.
\]
For $L^\infty$ potentials, Theorem
\ref{thm:meas-ucp} implies that
the zero set of a nontrivial state has measure zero, so the quotient recovers
$q$ a.e. For continuous potentials, the weak UCP in Theorem \ref{thm:one-state} shows that $\{u\ne0\}$ is dense in
the noncontact set; the quotient recovers $q$ on this dense set and
continuity determines it everywhere in the noncontact set.
\end{remark}

\section{A geometric coverage theorem}
\label{sec: geometric coverage}

The preceding result recovers $q$ only where at least one measured state lies
strictly above the obstacle $\psi$.  We now prove that, under a comparison principle,
a countable one-parameter family of exterior data exposes all of $\Omega$ up
to a null set.

\begin{lemma}[Weak maximum principle]
\label{lem:max}
Let $\Omega\subset\mathbb R^n$ be an open bounded set, $0<s<1$, and let
$q\in L^\infty(\Omega)$ satisfy $q\ge0$ a.e. in $\Omega$.  Let
$f\in H^s(\mathbb R^n)$ satisfy $f\ge0$ q.e. in $\R^n\setminus \Omega$. Let $h\in\Hs$
(weakly) solve
\[
  \begin{cases}
      ((-\Delta)^s+q)h=0\quad&\text{in }\Omega,\\
      h=f\quad&\text{in }\R^n\setminus\Omega.
  \end{cases}
\]
Then $h\ge0$ q.e. in $\R^n$.
\end{lemma}

\begin{proof}
The negative part $h_{-}=\max\{-h,0\}$ belongs to $\Htil(\Omega)$ because the
exterior datum is nonnegative; equivalently, the quasicontinuous
representative of $h_{-}$ vanishes q.e.~in $\Omega_e$ (see \cite[Theorem~10.1.1]{Adams-Hedberg}). Testing the equation
with $h_{-}$ gives $\Bq(h,h_{-})=0$. By the lattice inequality \eqref{eq: lattice energy inequalities}, applied to $-h$,
\[
  \mathcal E_s(h,h_-)
  =
  -\mathcal E_s(-h,h_-)=-\mathcal E_s(-h,(-h)_+)
  \le
  -\mathcal E_s(h_-,h_-).
\]
Moreover, since $q\ge0$,
\[
  \int_\Omega qhh_-\,dx
  =
  -\int_\Omega qh_-^2\,dx.
\]
Therefore
\[
  0=\Bq(h,h_-)
  \le
  -\mathcal E_s(h_-,h_-)-\int_\Omega qh_-^2\,dx
  =
  -\Bq(h_-,h_-).
\]
Hence $\Bq(h_-,h_-)\le0$. Coercivity of $\Bq$ on $\Htil(\Omega)$ yields
$h_-=0$, and therefore $h\ge0$ q.e. in $\mathbb R^n$.
\end{proof}

\begin{lemma}[Strong positivity]
\label{lem:strong-positivity}
Let $\Omega\subset\mathbb R^n$ be an open bounded set, $0<s<1$, and let
$q\in L^\infty(\Omega)$ satisfy $q\ge0$ a.e.~in $\Omega$.  Let
$f\in H^s(\mathbb R^n)$ satisfy $f\ge0$ q.e. in $\Omega_e$ and
$f\not\equiv0$ in $\Omega_e$.  If $h\in H^s(\mathbb R^n)$ solves
\[
  \begin{cases}
      ((-\Delta)^s+q)h=0\quad&\text{in }\Omega,\\
      h=f\quad&\text{in }\R^n\setminus\Omega.
  \end{cases}
\]
then $h>0$ a.e. in $\Omega$.
\end{lemma}

\begin{proof}
By Lemma \ref{lem:max}, $h\ge0$ q.e. in $\mathbb R^n$.  Since
\[
  (-\Delta)^s h=-qh\quad\text{in }\Omega,
\]
the strong maximum principle for nonlocal equations of the form
$Ih=c(x)h$, applied with $I=(-\Delta)^s$ and $c=-q\in L^\infty(\Omega)$,
implies that either $h>0$ a.e.~in $\Omega$ or $h\equiv0$ in
$\mathbb R^n$; see \cite[Theorem 1.1]{Jarohs-Weth-SMP}. The second
alternative is impossible because $h=f$ q.e.~in $\Omega_e$ and
$f\not\equiv0$ there. Note that in the case $s\geq 1/4$, the same conclusion also follows from the measurable UCP, Theorem~\ref{thm:meas-ucp}, instead of the strong maximum principle.
\end{proof}

\begin{lemma}[Obstacle state dominates the unconstrained state]
\label{lem:obstacle-dominates}
Let $\Omega\subset\mathbb R^n$ be an open bounded set, $0<s<1$, and let
$q\in L^\infty(\Omega)$ satisfy $q\ge0$ a.e.~in $\Omega$.  Let
$f_0\in H^s(\mathbb R^n)$, $\psi$ be quasicontinuous in $\Omega$, and
let $t>0$. Let $h\in H^s(\mathbb R^n)$ solve
\[
  \begin{cases}
      ((-\Delta)^s+q)h=0\quad&\text{in }\Omega,\\
      h=f_0\quad&\text{in }\R^n\setminus\Omega.
  \end{cases}
\]
Assume that $\K_{t f_0,\psi}$ is nonempty.  Let $u_t$ be the obstacle
solution with exterior datum $t f_0$, obstacle $\psi$, and potential $q$ (see Theorem~\ref{thm:obstacle-wp}).
Then
\[
  u_t\ge t h\quad \qe\text{ in }\Omega.
\]
\end{lemma}

\begin{proof}
Set $w=(th-u_t)_+\in H^s(\R^n)$.  Since $u_t=tf_0=th$ q.e.~in $\R^n\setminus\Omega$, we have
$w\in\Htil(\Omega)$.  The competitor $\max\{u_t,th\}=u_t+w\in H^s(\R^n)$ is admissible
because it lies above $u_t\ge\psi$ and has exterior value $t f_0$.  Hence the variational inequality \eqref{eq: variational inequality} yields
\[
  \Bq(u_t,w)\ge0.
\]
The linear equation for $th$ gives $\Bq(th,w)=0$, so
\[
  \Bq(th-u_t,w)\le0.
\]
By \eqref{eq: lattice energy inequalities} and the identity
\[
  \int_\Omega q(th-u_t)w\,dx=\int_\Omega qw^2\,dx,
\]
\[
  \Bq(th-u_t,w)\ge \Bq(w,w).
\]
Thus, coercivity of $\Bq$ on $\widetilde H^s(\Omega)$ gives $w=0$. This is the desired comparison.
\end{proof}

\begin{proposition}[Coverage by scaled positive data]
\label{prop:coverage}
Let $\Omega\subset\mathbb R^n$ be an open bounded set, $0<s<1$, and let
$q\in L^\infty(\Omega)$ satisfy $q\ge0$ a.e. in $\Omega$.  Let
$\psi$ be quasicontinuous in $\Omega$.  Let
$f_0\in\Hs$ satisfy $f_0\ge0$ q.e. in $\Omega_e$, be supported in
$\Omega_e$, and not be identically zero.  Assume
that, for every $t\in\mathbb Q_+$, the admissible set $\K_{t f_0,\psi}$ is
nonempty, and let $u_t$ be the corresponding obstacle solution.  If the
quasicontinuous representative of $\psi$ is finite a.e. in $\Omega$, then
\[
  \Omega
  =
  \bigcup_{t\in\mathbb Q_+}\{u_t>\psi\}
  \quad\text{up to a Lebesgue null set}.
\]
\end{proposition}

\begin{proof}
Let $h\in H^s(\R^n)$ be the solution \[
  \begin{cases}
      ((-\Delta)^s+q)h=0\quad&\text{in }\Omega,\\
      h=f_0\quad&\text{in }\R^n\setminus\Omega.
  \end{cases}
\]
Then Lemma~\ref{lem:strong-positivity} gives $h>0$ a.e. in $\Omega$. Now suppose that the uncovered set
\[
  A \vcentcolon =\Omega\setminus\bigcup_{t\in\mathbb Q_+}\{u_t>\psi\}
\]
has positive measure. Since $\psi$ is finite a.e., there is $m>0$ such that
\[
    A_m \vcentcolon =A\cap\{|\psi|\le m\}
\]
has positive measure.  For every rational $t>0$,
one has $u_t\le\psi$ on $A_m$ outside a set of zero $H^s$-capacity, while
Lemma \ref{lem:obstacle-dominates} gives $u_t\ge th$ q.e. in $\Omega$.
Since the rational parameters are countable and sets of zero $H^s$-capacity
have zero Lebesgue measure (see Lemma~\ref{lemma: comparison to Lebesgue measure}),
\[
  t h\le m\quad\text{a.e. on }A_m
  \qquad \forall t\in\mathbb Q_+.
\]
Letting $t\to\infty$ through rational values yields $h=0$ a.e. on $A_m$,
contradicting $h>0$ a.e. in $\Omega$.
Hence $A$ is null.
\end{proof}

\begin{theorem}[Global uniqueness from a countable scaled family]
\label{thm:global}
Let $\Omega\subset\mathbb R^n$ be an open bounded set, $0<s<1$, and let
$q_1,q_2\in L^\infty(\Omega)$ satisfy $q_j\ge0$ a.e. in $\Omega$,
$j=1,2$.  Assume either that $s\in[1/4,1)$, so that Theorem
\ref{thm:meas-ucp} applies, or that $q_1,q_2$ have continuous representatives
in $\Omega$.  Let $W_1,W_2\subset\Omega_e$ be nonempty open sets, let
\[
  0\le f_0\in C_c^\infty(W_1)
\]
be not identically zero, and assume $\psi\in C_c^\infty(\Omega)$.
If
\[
  \Lambda_{q_1,\psi}(t f_0)|_{W_2}
  =
  \Lambda_{q_2,\psi}(t f_0)|_{W_2}
  \qquad\forall t\in\mathbb Q_+
\]
then
\[
  q_1=q_2\quad\text{a.e. in }\Omega.
\]
\end{theorem}

\begin{proof}
The admissible sets are nonempty for all $t\in\mathbb Q_+$.  Indeed, choose
$\chi\in C_c^\infty(\Omega)$ with $\chi\ge1$ on $\supp\psi$ and
$\chi\ge0$.  For $M\ge\|\psi\|_{L^\infty(\Omega)}$, the function
$t f_0+M\chi$ has exterior value $t f_0$ and lies above $\psi$ in $\Omega$.

For each $t\in\mathbb Q_+$, Theorem \ref{thm:one-state}, using the
measurable-set branch when $s\in[1/4,1)$ and the continuous-potential branch
otherwise, shows that the two obstacle states coincide and that
$q_1=q_2$ a.e. in the common noncontact set $\{u_t>\psi\}$.  The additional
openness hypotheses in that theorem hold for $q_1$: since
$\psi\in C_c^\infty(\Omega)$ and
$f_0\in C_c^\infty(W_1)$, the function
\[
    \Psi_t:=\psi+t f_0\in C_c^{\infty}(\R^n)
\]
satisfies
$\Psi_t-tf_0=\psi\in\Htil(\Omega)$ and
$L_{q_1}\Psi_t\in L^\infty(\Omega)$.  Also
$L_{q_1}(t f_0)\in L^\infty(\Omega)$, because $f_0$ is smooth and compactly
supported in the exterior open set $W_1$. Thus Lemma \ref{lem:Lewy--Stampacchia} applies. Moreover, we used that by assumption $t f_0\not\equiv0$.

Now, Proposition
\ref{prop:coverage}, applied to $q_1$, says that the union of the
noncontact sets $\{u_t>\psi\}$, $t\in \mathbb Q_+$, covers $\Omega$ up to a null set. Hence $q_1=q_2$ a.e. in
$\Omega$.
\end{proof}

\begin{remark}[Relation with linearized obstacle methods]
\label{rem:linearized-strategy}
The proof of Theorem~\ref{thm:global} does not use differentiability of the
obstacle solution map. This is an important difference from the local
linearized free-boundary strategy, which we pursue elsewhere for the local
Schrödinger equation.  In that approach, one studies
derivatives of the obstacle solution with respect to exterior or obstacle
perturbations.  Formally, at a background solution $u_0$ with noncontact set
$D=\{u_0>\psi_0\}$ and contact set $K=\Omega\setminus D$, exterior variations
lead to solutions of
\[
  ((-\Delta)^s+q)w=0 \quad\text{in }D,
  \qquad
  w=0 \quad\text{on }K
\]
in the capacitary weak sense, while obstacle variations prescribe the
corresponding capacitary trace on $K$.

Making this linearized strategy rigorous requires additional assumptions
which are not needed in the proof of Theorem~\ref{thm:global}.  One needs,
for instance, a Mignot-type differentiability theorem for the fractional
obstacle variational inequality, a capacitary description of the tangent cone,
and a strict-complementarity or critical-cone collapse condition ensuring that
the linearized directions vanish on the contact set.  Under such hypotheses
one obtains a nonlocal analogue of the classical linearized free-boundary
argument: exterior perturbations solve a fractional Schrödinger equation on
the exposed quasiopen set $D$, whereas obstacle perturbations prescribe
capacitary data on the whole contact set $K$.

Thus the linearized picture remains conceptually useful, but it does not
strengthen the main uniqueness theorem proved here.  The advantage of the
present argument is precisely that it avoids these differentiability and
strict-complementarity hypotheses.
\end{remark}

	\medskip

    \noindent\textbf{Acknowledgments.} 
	\begin{itemize}
			\item G.~Uhlmann is partially supported by NSF.
			\item P.~Zimmermann is supported by the Swiss National Science Foundation (SNSF) under Grant Nos.~214500 and 239130.
		\end{itemize}

	\bibliographystyle{alpha}
	
	\bibliography{refs} 

@book{Adams-Hedberg,
  author    = {Adams, D. R. and Hedberg, L. I.},
  title     = {{Function Spaces and Potential Theory}},
  series    = {Grundlehren der Mathematischen Wissenschaften},
  volume    = {314},
  publisher = {Springer},
  address   = {Berlin},
  year      = {1996}
}

@article{calderon2006inverse,
  title={On an inverse boundary value problem},
  author={Calder{\'o}n, Alberto P},
  journal={Computational \& Applied Mathematics},
  volume={25},
  number={2-3},
  pages={133--138},
  year={2006},
  publisher={SciELO Brasil}
}

@book {evans-mt,
	AUTHOR = {Evans, Lawrence C. and Gariepy, Ronald F.},
	TITLE = {Measure theory and fine properties of functions},
	SERIES = {Textbooks in Mathematics},
	EDITION = {Revised},
	PUBLISHER = {CRC Press, Boca Raton, FL},
	YEAR = {2015},
	PAGES = {xiv+299},
	ISBN = {978-1-4822-4238-6},
	MRCLASS = {28-01},
	MRNUMBER = {3409135},
}

@article {RZ-unbounded,
	AUTHOR = {Railo, Jesse and Zimmermann, Philipp},
	TITLE = {Fractional {C}alder\'{o}n problems and {P}oincar\'{e}
	inequalities on unbounded domains},
	JOURNAL = {J. Spectr. Theory},
	FJOURNAL = {Journal of Spectral Theory},
	VOLUME = {13},
	YEAR = {2023},
	NUMBER = {1},
	PAGES = {63--131},
	ISSN = {1664-039X,1664-0403},
	MRCLASS = {35R30 (26A33 42B37 46F12)},
	MRNUMBER = {4620353},
	DOI = {10.4171/jst/444},
	URL = {https://doi.org/10.4171/jst/444},
}

@article{Silvestre-obstacle,
  author  = {Silvestre, L.},
  title   = {{The regularity of the obstacle problem for a fractional power of the Laplace operator}},
  journal = {Comm. Pure Appl. Math.},
  volume  = {60},
  number  = {1},
  pages   = {67--112},
  year    = {2007}
}

@book{fernandez2024integro,
  title     = {{Integro-Differential Elliptic Equations}},
  author    = {Fern{\'a}ndez-Real, Xavier and Ros-Oton, Xavier},
  publisher = {Springer},
  year      = {2024}
}

@phdthesis{Silvestre-thesis,
  author = {Silvestre, L. E.},
  title  = {{Regularity of the Obstacle Problem for a Fractional Power of the Laplace Operator}},
  school = {The University of Texas at Austin},
  year   = {2005}
}

@article{Silvestre-Holder,
  author  = {Silvestre, L.},
  title   = {{H{\"o}lder estimates for solutions of integro-differential equations like the fractional Laplace}},
  journal = {Indiana Univ. Math. J.},
  volume  = {55},
  number  = {3},
  pages   = {1155--1174},
  year    = {2006}
}

@article{DNPV-Hitchhiker,
  author  = {Di Nezza, E. and Palatucci, G. and Valdinoci, E.},
  title   = {{Hitchhiker's guide to the fractional Sobolev spaces}},
  journal = {Bull. Sci. Math.},
  volume  = {136},
  number  = {5},
  pages   = {521--573},
  year    = {2012}
}

@article{Fall-Felli-UCP,
  author  = {Fall, M. M. and Felli, V.},
  title   = {{Unique continuation property and local asymptotics of solutions to fractional elliptic equations}},
  journal = {Comm. Partial Differential Equations},
  volume  = {39},
  pages   = {354--397},
  year    = {2014}
}

@article{GSU-fractional,
  author  = {Ghosh, T. and Salo, M. and Uhlmann, G.},
  title   = {{The Calder{\'o}n problem for the fractional Schr{\"o}dinger equation}},
  journal = {Anal. PDE},
  volume  = {13},
  number  = {2},
  pages   = {455--475},
  year    = {2020}
}

@book {Variational-Methods,
	AUTHOR = {Struwe, Michael},
	TITLE = {Variational methods},
	SERIES = {Ergebnisse der Mathematik und ihrer Grenzgebiete. 3. Folge. A
	Series of Modern Surveys in Mathematics [Results in
	Mathematics and Related Areas. 3rd Series. A Series of Modern
	Surveys in Mathematics]},
	VOLUME = {34},
	EDITION = {Fourth},
	NOTE = {Applications to nonlinear partial differential equations and
	Hamiltonian systems},
	PUBLISHER = {Springer-Verlag, Berlin},
	YEAR = {2008},
	PAGES = {xx+302},
	ISBN = {978-3-540-74012-4},
	MRCLASS = {49-02 (34C25 35A15 35F20 37J45 47J30 49J10 58E05)},
	MRNUMBER = {2431434},
}

@misc{GRSU-single,
  author       = {Ghosh, T. and R{\"u}land, A. and Salo, M. and Uhlmann, G.},
  title        = {{Uniqueness and reconstruction for the fractional Calder{\'o}n problem with a single measurement}},
  howpublished = {preprint, arXiv:1801.04449},
  year         = {2018}
}

@article{Jarohs-Weth-SMP,
  author  = {Jarohs, S. and Weth, T.},
  title   = {{On the strong maximum principle for nonlocal operators}},
  journal = {Math. Z.},
  volume  = {293},
  pages   = {81--111},
  year    = {2019}
}

@misc{Kow-Kimura-LS,
  author       = {Kow, P.-Z. and Kimura, M.},
  title        = {{The Lewy-Stampacchia inequality for the fractional Laplacian and its application to anomalous unidirectional diffusion equations}},
  howpublished = {preprint, arXiv:1909.00588},
  year         = {2019}
}

@misc{Lo-Rodrigues-LS,
  author       = {Lo, C. W. K. and Rodrigues, J. F.},
  title        = {{On a class of nonlocal obstacle type problems related to the distributional Riesz fractional derivative}},
  howpublished = {preprint, arXiv:2101.06863},
  year         = {2021}
}

@article{ROS-Dirichlet,
  author  = {Ros-Oton, X. and Serra, J.},
  title   = {{The Dirichlet problem for the fractional Laplacian: regularity up to the boundary}},
  journal = {J. Math. Pures Appl.},
  volume  = {101},
  number  = {3},
  pages   = {275--302},
  year    = {2014}
}

@book{schaeferbanach,
  title     = {{Banach Lattices and Positive Operators}},
  author    = {Schaefer, Helmut H.},
  publisher = {Springer},
  address   = {Berlin},
  year      = {1974}
}

@incollection{mosco2006implicit,
  title={Implicit variational problems and quasi variational inequalities},
  author={Mosco, Umberto},
  booktitle={Nonlinear Operators and the Calculus of Variations: Summer School Held in Bruxelles 8--19 September 1975},
  pages={83--156},
  year={2006},
  publisher={Springer}
}

@article{Rodrigues-Teymurazyan-LS,
  title={On the two obstacles problem in Orlicz--Sobolev spaces and applications},
  author={Rodrigues, Jos{\'e} Francisco and Teymurazyan, Rafayel},
  journal={Complex Variables and Elliptic Equations},
  volume={56},
  number={7-9},
  pages={769--787},
  year={2011},
  publisher={Taylor \& Francis}
}

@article{KLW2022,
	AUTHOR = {Kow, Pu-Zhao and Lin, Yi-Hsuan and Wang, Jenn-Nan},
	TITLE = {The {C}alder\'{o}n problem for the fractional wave equation:
	uniqueness and optimal stability},
	JOURNAL = {SIAM J. Math. Anal.},
	FJOURNAL = {SIAM Journal on Mathematical Analysis},
	VOLUME = {54},
	YEAR = {2022},
	NUMBER = {3},
	PAGES = {3379--3419},
	ISSN = {0036-1410,1095-7154},
	MRCLASS = {35B35 (35R11 35R30)},
	MRNUMBER = {4434352},
	MRREVIEWER = {Wei\ Lian},
	DOI = {10.1137/21M1444941},
	URL = {https://doi.org/10.1137/21M1444941},
}

@article{zimmermann2024-viscous-wave,
title = {Calder\'on problem for nonlocal viscous wave equations: {U}nique determination of linear and nonlinear perturbations},
journal = {Revista de la Real Academia de Ciencias Exactas, Físicas y Naturales. Serie A. Matemáticas},
volume = {120},
NUMBER = {2},
pages = {31},
year = {2026},
doi = {https://doi.org/10.1007/s13398-025-01822-0},
url = {https://rdcu.be/eYhRi},
author = {Philipp Zimmermann}
}

@article{fu2026calderon,
  title={The Calder{\'o}n problem for third order nonlocal wave equations with time-dependent nonlinearities and potentials},
  author={Fu, Song-Ren and Yu, Yongyi and Zimmermann, Philipp},
  journal={Journal of Differential Equations},
  volume={463},
  pages={114164},
  year={2026},
  publisher={Elsevier}
}

@article {GSU20,
	AUTHOR = {Ghosh, Tuhin and Salo, Mikko and Uhlmann, Gunther},
	TITLE = {The {C}alder\'{o}n problem for the fractional {S}chr\"{o}dinger
	equation},
	JOURNAL = {Anal. PDE},
	FJOURNAL = {Analysis \& PDE},
	VOLUME = {13},
	YEAR = {2020},
	NUMBER = {2},
	PAGES = {455--475},
	ISSN = {2157-5045},
	MRCLASS = {35R11 (26A33 35J10 35J70 35R30)},
	MRNUMBER = {4078233},
	DOI = {10.2140/apde.2020.13.455},
	URL = {https://doi.org/10.2140/apde.2020.13.455},
}

@article{RZ2022LowReg,
	title={Low regularity theory for the inverse fractional conductivity problem},
	author={Railo, Jesse and Zimmermann, Philipp},
	journal={Nonlinear Analysis},
	volume={239},
	pages={113418},
	year={2024},
	publisher={Elsevier}
}

@article{RS17,
	AUTHOR = {R\"{u}land, Angkana and Salo, Mikko},
	TITLE = {The fractional {C}alder\'{o}n problem: low regularity and
	stability},
	JOURNAL = {Nonlinear Anal.},
	FJOURNAL = {Nonlinear Analysis. Theory, Methods \& Applications. An
	International Multidisciplinary Journal},
	VOLUME = {193},
	YEAR = {2020},
	PAGES = {111529, 56},
	ISSN = {0362-546X},
	MRCLASS = {35R30 (35B35 35B65 35R11 46E35)},
	MRNUMBER = {4062981},
	DOI = {10.1016/j.na.2019.05.010},
	URL = {https://doi.org/10.1016/j.na.2019.05.010},
}

@article{ghosh2021non,
	AUTHOR = {Ghosh, Tuhin},
	TITLE = {A non-local inverse problem with boundary response},
	JOURNAL = {Rev. Mat. Iberoam.},
	FJOURNAL = {Revista Matem\'{a}tica Iberoamericana},
	VOLUME = {38},
	YEAR = {2022},
	NUMBER = {6},
	PAGES = {2011--2032},
	ISSN = {0213-2230},
	MRCLASS = {35R30 (35R11)},
	MRNUMBER = {4516180},
	DOI = {10.4171/RMI/1323},
	URL = {https://doi.org/10.4171/RMI/1323},
}

@article{GU2021calder,
	title={The {C}alder\'{o}n problem for nonlocal operators},
	author={Ghosh, Tuhin and Uhlmann, Gunther},
	journal={arXiv:2110.09265},
	year={2021}
}

@article{zimmermann2023inverse,
	TITLE = {Inverse problem for a nonlocal diffuse optical tomography
	equation},
	author={Zimmermann, Philipp},
	JOURNAL = {Inverse Problems},
	FJOURNAL = {Inverse Problems. An International Journal on the Theory and
	Practice of Inverse Problems, Inverse Methods and Computerized
	Inversion of Data},
	VOLUME = {39},
	YEAR = {2023},
	NUMBER = {9},
	PAGES = {Paper No. 094001, 25},
	ISSN = {0266-5611,1361-6420},
	MRCLASS = {65M32},
	MRNUMBER = {4619274},
}

@article{cekic2020calderon,
	title={The {C}alder{\'o}n problem for the fractional {S}chr{\"o}dinger equation with drift},
	author={Cekic, Mihajlo and Lin, Yi-Hsuan and R{\"u}land, Angkana},
	journal={Cal. Var. Partial Differential Equations},
	volume={59},
	number={91},
	url                  = {https://link.springer.com/article/10.1007/s00526-020-01740-6},
	year={2020}
}

@article {ruland2018exponential,
	AUTHOR = {R\"{u}land, Angkana and Salo, Mikko},
	TITLE = {Exponential instability in the fractional {C}alder\'{o}n problem},
	JOURNAL = {Inverse Problems},
	FJOURNAL = {Inverse Problems. An International Journal on the Theory and
	Practice of Inverse Problems, Inverse Methods and Computerized
	Inversion of Data},
	VOLUME = {34},
	YEAR = {2018},
	NUMBER = {4},
	PAGES = {045003, 21},
	ISSN = {0266-5611},
	MRCLASS = {35R30 (35B35 35R11 65J22)},
	MRNUMBER = {3774704},
	DOI = {10.1088/1361-6420/aaac5a},
	URL = {https://doi.org/10.1088/1361-6420/aaac5a},
}

@article {GRSU18,
	AUTHOR = {Ghosh, Tuhin and R\"{u}land, Angkana and Salo, Mikko and Uhlmann,
	Gunther},
	TITLE = {Uniqueness and reconstruction for the fractional {C}alder\'{o}n
	problem with a single measurement},
	JOURNAL = {J. Funct. Anal.},
	FJOURNAL = {Journal of Functional Analysis},
	VOLUME = {279},
	YEAR = {2020},
	NUMBER = {1},
	PAGES = {108505, 42},
	ISSN = {0022-1236},
	MRCLASS = {35R11},
	MRNUMBER = {4083776},
	MRREVIEWER = {Vincenzo Ambrosio},
	DOI = {10.1016/j.jfa.2020.108505},
	URL = {https://doi.org/10.1016/j.jfa.2020.108505},
}

@article{feizmohammadi2021fractional_closed,
	title={Fractional {C}alder\'on' problem on a closed {R}iemannian manifold},
	author={Feizmohammadi, Ali},
	journal={arXiv preprint arXiv:2110.07500},
	year={2021}
}

@article{FKU24,
	title={Calder\'{o}n problem for fractional {S}chr\"{o}dinger operators on closed {R}iemannian manifolds},
	author={Feizmohammadi,Ali and Krupchyk, Katya and Uhlmann, Gunther},
	journal={arXiv preprint arXiv:2407.16866},
	year={2024}
}

@article{zimmermann2024optimalrungeapproximationdamped,
  title={Optimal Runge approximation for damped nonlocal wave equations and simultaneous determination results},
  author={Philipp Zimmermann},
  journal={Journal of Spectral Theory},
  year={2024},
  url={https://api.semanticscholar.org/CorpusID:274446033}
}

@article{feizmohammadi2024calder,
  title={Calder$\backslash$'$\{$o$\}$ n problem for fractional Schr$\backslash$"$\{$o$\}$ dinger operators on closed Riemannian manifolds},
  author={Feizmohammadi, Ali and Krupchyk, Katya and Uhlmann, Gunther},
  journal={arXiv preprint arXiv:2407.16866},
  year={2024}
}

@article{ruland2021single,
  title={On single measurement stability for the fractional Calder{\'o}n problem},
  author={Rüland, Angkana},
  journal={SIAM Journal on Mathematical Analysis},
  volume={53},
  number={5},
  pages={5094--5113},
  year={2021},
  publisher={SIAM}
}
	
\end{document}